\documentclass[11pt]{amsart}
\usepackage{amsmath,amsfonts,amssymb,amsthm}
\usepackage{mathrsfs}
\usepackage[all,cmtip]{xy}

\usepackage[left=0.88in,right=0.88in,top=0.89in,bottom=0.89in]{geometry}

\newtheorem{theorem}{Theorem}[section]
\newtheorem{lemma}[theorem]{Lemma}
\newtheorem{corollary}[theorem]{Corollary}
\newtheorem{proposition}[theorem]{Proposition}

\newtheorem{problem}[theorem]{Problem}

\theoremstyle{definition}
\newtheorem{question}[theorem]{Question}
\newtheorem{fact}[theorem]{Fact}
\newtheorem{example}[theorem]{Example}
\newtheorem{definition}[theorem]{Definition}
\newtheorem{remark}[theorem]{Remark}

\newcommand{\Z}{{\mathbb Z}}
\newcommand{\N}{{\mathbb N}}
\newcommand{\R}{{\mathbb R}}

\newcommand{\cont}{{\mathfrak c}}

\def\kal#1{\mathcal{K}_{#1}}  
\def\skal#1{\overline{\mathcal{K}}_{#1}} 
\def\lkal#1{l\kal{#1}} 
\def\slkal#1{\overline{l\mathcal{K}}_{#1}} \def\t#1#2{\mathscr{T}_{{#1},{#2}}}   
\def\td#1{\mathscr{T}_{#1}}

\def\hull#1{\left\langle{#1}\right\rangle}
\def\nbh#1#2#3{\mathscr{O}(#1,#2,#3)}

\title[Direct sums and products in topological groups and vector spaces]{Direct sums and products in topological groups and vector spaces}

\author[D. Dikranjan]{Dikran Dikranjan}
\address[Dikran Dikranjan]{Dipartimento di Matematica e Informatica, Universit\`{a} di Udine, Via delle Scienze  206, 33100 Udine, Italy}
\email{dikran.dikranjan@uniud.it}

\author[D. Shakhmatov]{Dmitri Shakhmatov}
\address[Dmitri Shakhmatov]{Division of Mathematics, Physics and Earth Sciences\\
Graduate School of Science and Engineering\\
Ehime University, Matsuyama 790-8577, Japan}
\email{dmitri.shakhmatov@ehime-u.ac.jp}

\author[J. Sp\v{e}v\'ak]{Jan Sp\v{e}v\'ak}
\address[Jan Sp\v{e}v\'ak]{Department of Mathematics\\ Faculty of Science\\ J. E. Purkyne
University, \v{C}esk\'{e} ml\'{a}de\v{z}e 8,
400 96 \'{U}st\'{i} nad Labem\\
Czech Republic} \email{jan.spevak@ujep.cz}


\thanks{The first named author was supported by the grant ``Progetti di Eccellenza 2011/12" of Fondazione CARIPARO.
The second named author was partially supported by the Grant-in-Aid for Scientific Research~(C) No.~26400091 by the Japan Society for the Promotion of Science (JSPS).
The third named author was supported by the grant P201-12-P724 of GA\v{C}R}

\begin{document}

\begin{abstract}

We call a subset $A$ of an abelian topological group $G$: (i) {\em absolutely Cauchy summable\/} provided that for every open neighbourhood $U$ of  $0$  one can find a finite set $F\subseteq A$ such that the subgroup generated by $A\setminus F$ is contained in $U$; (ii) {\em absolutely summable\/} if, for every family $\{z_a:a\in A\}$ of integer numbers, there exists $g\in G$ such that the net $\left\{\sum_{a\in F} z_a a: F\subseteq A\mbox{ is finite}\right\}$ converges to $g$; (iii) {\em topologically independent\/} provided that $0\not \in A$ and for every neighbourhood $W$ of  $0$  there exists a neighbourhood $V$ of  $0$  such that, for every finite set $F\subseteq A$ and each set $\{z_a:a\in F\}$ of integers, $\sum_{a\in F}z_aa\in V$ implies that $z_aa\in W$ for all $a\in F$. We prove that: (1) an abelian topological group contains a direct product (direct sum) of $\kappa$-many non-trivial topological groups if and only if it contains a topologically independent, absolutely (Cauchy) summable subset of cardinality $\kappa$; (2) a topological vector space contains $\R^{(\N)}$ as its subspace if and only if it has an infinite absolutely Cauchy summable set; (3) a topological vector space contains $\R^{\N}$ as its subspace if and only if it has an $\R^\N$ multiplier convergent series of non-zero elements. We answer a question of Hu\v{s}ek and  generalize results by Bessaga-Pelczynski-Rolewicz, Dominguez-Tarieladze and Lipecki.
\end{abstract}


\dedicatory{In memory of Nigel J. Kalton}

\maketitle

\def\rhull#1{\hull{#1}_\R}

The symbols $\N$, $\Z$ and $\R$ stay for the sets of natural numbers, integers and real numbers respectively equipped with their usual 
algebraic  and topological structures. 

Let $G$ be an abelian  group. By $0_G$ we denote the identity element of $G$, and the subscript is often omitted when there is no danger of confusion. 
Given a subset $A$ of $G$, the symbol $\hull{A}$ stays for the subgroup of $G$ generated by $A$.
For every element $a\in G$, we use $\hull{a}$  to denote the cyclic subgroup 
of $G$ generated by $a$; that is, we use $\hull{a}$ instead of $\hull{\{a\}}$ for brevity.  When $G$ is a topological group,
we always consider $\hull{a}$  with the subspace topology inherited from $G$.
If  $\{A_i:i\in I\}$ is a finite family of subsets of $G$, then we let
$$
\sum_{i\in I} A_i=\begin{cases} \left\{\sum_{i\in I}a_i: (\forall i\in I) a_i \in A_i\right\}\mbox{ in case }I \ne \emptyset\\
0 \mbox{ in case }I =\emptyset.
\end{cases}
$$

Let $I$ be a non-empty set and let $\{G_i:i\in I\}$ be a family of groups. 
As usual, $\prod_{i\in I} G_i$ denotes the {\em direct product\/} of this family consisting of all elements $\{g_i\}_{i\in I}$ such that $g_i\in G_i$ for every $i\in I$. It becomes a group under coordinate-wise multiplication.
The {\em direct sum\/} $\bigoplus_{i\in I} G_i$ of the family $\{G_i:i\in I\}$ is a subgroup of the direct product $\prod_{i\in I} G_i$ consisting of those elements $\{g_i\}_{i\in I}\in \prod_{i\in I} G_i$
for which the set $\{i\in I: g_i\not = 0_{G_i}\}$ is finite. If $X_i\subseteq G_i$ for every $i \in I$, we denote by  
$\bigoplus_{i\in I} X_i$  the set  $\{\{g_i\}_{i\in I}\in \bigoplus_{i\in I} G_i: g_i\in X_i \mbox{ for all } i\in I\}$. If $G_i=G$ for all $i\in I$, then following  the
standard practice, the direct product $\prod_{i\in I} G_i$ is denoted by $G^I$ and the correspondent direct sum is denoted by $G^{(I)}$.
 
An {\em NSS} group is a topological group which has a neighbourhood of the identity that contains no non-trivial subgroup.

Let $A$ be a subset of  a vector space $E$. The subspace of $E$ generated by $A$ is denoted by $\rhull{A}$. If $a\in E$, then we use $\rhull{a}$ instead of $\rhull{\{a\}}$ for brevity. Clearly, $\rhull{a}=\{r a: r\in\R\}$ is the ``line'' passing through $a$ and $0_E$.

{\bf All topological groups are assumed to be Hausdorff.\/}

\section{Introduction}

The aim of this paper is to study the following fundamental question: 
\begin{question}
\label{our:aim}
Let $G$ be a topological group.
\begin{itemize}
\item[(i)] When does $G$ contain an (infinite) direct sum of non-trivial topological groups? 
\item[(ii)] When does $G$ contain an (infinite) direct product of non-trivial topological groups? 
\end{itemize}
\end{question}
Clearly, a topological group contains (a subgroup topologically isomorphic to) a direct product (a direct sum) of $\kappa$-many non-trivial topological groups if and only if it contains a direct product (a direct sum) of $\kappa$-many non-trivial cyclic topological groups.
Therefore, we can reduce Question \ref{our:aim} to the question of when does a topological group contain a direct product (a direct sum) of non-trivial {\em cyclic\/} topological groups.
 Since both a direct sum and a direct product of cyclic groups are necessarily abelian, by passing to a subgroup of the group $G$ in Question \ref{our:aim} if necessary, we may assume, without loss of generality, that $G$ itself is abelian. This is why for the rest of this paper
{\bf we shall assume that all (topological) groups are abelian\/}.

In order to tackle Question \ref{our:aim}, let us to introduce a relevant notation. Let $G$ be a group and let $A\subseteq G\setminus \{0\}$. Then there exists a unique group homomorphism 
\begin{equation}
\label{eq:K_A}
\kal{A}: \bigoplus _{a\in A} \hull{a}\to G
\end{equation}
which extends each natural inclusion map $\hull{a}\to G$ for $a\in A$. The set $A$ is said to be {\em independent\/} if $\kal{A}$ is monomorphism.

We call the map $\kal{A}$ as in \eqref{eq:K_A} the {\em Kalton map associated with $A$\/}, in memory of Nigel Kalton whose idea presented in \cite{kal}  inspired us to write this manuscript.

When $G$ is a topological group, the cyclic subgroup $\hull{a}$ of $G$ generated by an element $a\in A$ inherits the subspace topology from $G$. Therefore, one can consider the Tychonoff product topology on the direct product 
$$
P_A=\prod_{a\in A}\hull{a},
$$
where each $\hull{a}$ carries the subgroup topology inherited from $G$. Since the direct sum 
\begin{equation}
\label{eq:S_A}
S_A=\bigoplus _{a\in A} \hull{a}
\end{equation}
is naturally identified with a subgroup of $P_A$, it carries the subgroup topology. Now both the domain \eqref{eq:S_A} and the range $G$ of the map \eqref{eq:K_A} are equipped with group topologies, thereby allowing one to talk about {\em topological} properties of the homomorphism $\kal{A}$ such as its continuity and openness. If $\kal{A}$ is a monomorphism which is both continuous and an open map onto its image $\kal{A}(S_A)=\hull{A}$,
then $\kal{A}$ is a topological isomorphism between $S_A$ and the subgroup $\hull{A}$
of $G$, thereby yielding (a subgroup topologically isomorphic to) a direct sum of $|A|$-many non-trivial cyclic groups in $G$. The converse also obviously holds. 
\begin{fact}
\label{fact}
For a topological group $G$, the following conditions are equivalent:
\begin{itemize}
\item[(i)] $G$ contains a direct sum of $\kappa$-many non-trivial topological groups;
\item[(ii)] there exists a set $A\subseteq G\setminus\{0\}$ such that
$|A|=\kappa$ and $\kal{A}$ is a continuous monomorphism which is an open map onto its image $\kal{A}(S_A)$, or equivalently, $\kal{A}$ is a topological isomorphism between $S_A$ and the subgroup $\kal{A}(S_A)=\hull{A}$
of $G$.
\end{itemize}
\end{fact}

This fact shows that Question \ref{our:aim}~(i) transforms to the following problem:

\begin{problem}
\label{when:Kalton:is:an:top:isomorhism}
Given a subset $A$ of a topological group $G$ such that $0\not\in G$, find the necessary and sufficient conditions on $A$ to make the Kalton map $\kal{A}:S_A\to G$ a topologically isomorphic embedding.
\end{problem}

We completely resolve this problem
in two steps. 
First, in Section \ref{sec:abs:cauch:summable:sets} we introduce the notion of an {\em absolutely Cauchy summable set} in a topological group which generalizes the concept of an $f_\omega$-Cauchy summable sequence introduced and studied extensively by the authors in~\cite{DSS_arxiv}. 
The importance of this notion becomes clear from
Theorem \ref{Proposition:Udine}
which states that 
the Kalton map $\kal{A}$ is continuous if and only if $A$ is absolutely Cauchy summable set. 
Second,in 
Section \ref{sec:top:indep:sets} we introduce the notion of a {\em topologically independent set} in a topological group which is a counterpart of the notion of an independent set in a group. Indeed, topologically independent sets are independent and these two notions coincide in discrete groups by Lemma \ref{lemma:top:ind:is:ind}.
While independent set already produces a direct sum in a group, the topological independence of a set $A$ is only a necessary condition for the Kalton map $\kal{A}$ to be a topologically isomorphic embedding; this fact
is established in Proposition \ref{Proposition:Matsuyama}~(ii). 
It turns out that the Kalton map $\kal{A}$ is a topologically isomorphic embedding if and only if 
$A$ is {\em both\/} topologically independent and absolutely Cauchy summable;
see Theorem \ref{Kalton:top:iso}. Indeed, topological independence of $A$ assures that $\kal{A}$ is an open isomorphism (taken as a map onto $\hull{A}$) while the absolute Cauchy summability 
is responsible for 
the continuity of $\kal{A}$. 
This resolves Problem \ref{when:Kalton:is:an:top:isomorhism} and thus,
Question \ref{our:aim}~(i) as well.

It is worth noting that neither of the two notions taken alone guarantees that 
the Kalton map $\kal{A}$ is a topological embedding. (In view of Theorem \ref{Kalton:top:iso} described above, this means that these two notions are logically independent of each other.)
Indeed, 
an example of a compact group in which every null sequence is absolutely Cauchy summable but the group does not contain even a product of two non-trivial groups is given in Remark \ref{remark:Zp}.
Similarly, 
every Schauder basis in a Banach space is topologically independent
by Proposition \ref{schauder:is:top:ind}, yet 
Banach spaces do not contain infinite direct sums; see Remark \ref{rem:on:Schauder}.

The step towards the question when a topological group contains a direct product of cyclic topological groups is now
straightforward. Given an absolutely Cauchy summable set $A$ in a topological group $G$, the Kalton map $\kal{A}$ is continuous. Thus, it continuously extends to a 
unique map $\skal{A}$ from the completion of $S_A$ to the completion of $G$. Since $P_A$ is a subset of the completion of $S_A$, if 
the image $\skal{A}(P_A)$ of $P_A$ under $\skal{A}$ is in $G$,
then $G$ contains a direct product $P_A$ of $|A|$-many non-trivial groups.
The converse is also true: If $G$ contains a direct product of $|A|$-many non-trivial groups, then it contains a subset $A$ such that $\skal{A}(P_A)\subseteq G$. 
Therefore, Question \ref{our:aim}~(ii)
transforms to the following problem:

\begin{problem}
\label{when:extended:Kalton:is:an:top:isomorhism}
Given a subset $A$ of a topological group $G$ such that $0\not\in G$
and $\kal{A}:S_A\to G$ is an isomorphic embedding,
find the necessary and sufficient conditions on $A$ equivalent to 
the inclusion $\skal{A}(P_A)\subseteq G$. 
\end{problem}

It  turns out that this happens precisely when the set $A$ is {\em absolutely summable\/} in $G$. This notion is introduced in Section \ref{sec:abs:summable:sets}, where we also show that it is a generalization of the concept of $f_\omega$-summable sequences studied in \cite{DSS_arxiv}. We prove that $\skal{A}\restriction{P_A}$ is a topologically isomorphic embedding into $G$ if and only if $A$ is absolutely summable and topologically independent; see Theorem \ref{characterization:of:topological:isomorphism:of:extended:Kalton:maps}.
Absolutely summable sets are absolutely Cauchy summable, and the converse 
holds in complete groups; see Proposition \ref{as:is:acs}.

In Section \ref{Sec:inf:dir:sums:in:top:groups} we provide our first applications of the above mentioned results. Namely, we  generalize the result of  Dominguez and Tarieladze from \cite{DT-private} by proving that every metric torsion free group such that each its cyclic subgroup is discrete is not NSS if and only if it contains as a subgroup the topological group $\Z^{(\N)}$; see Corollary \ref{sum:of:Z:inside}. As another example of various possible applications we also show that a complete metric torsion group is not NSS if and only if it contains a (compact) subgroup that is topologically isomorphic  to an infinite product of non-trivial finite cyclic groups; in particular, if the above group is not NSS, then it contains an infinite compact zero-dimensional subgroup (Theorem \ref{compact:zero:dim:inside}). This property was used in \cite{DS_Lie} to obtain a characterization of the Lie groups. 

If $A$ is a subset of a topological vector space, then the Kalton map $\kal{A}$ can be naturally extended to the {\em linear Kalton map $\lkal{A}$\/}; this is done in Section \ref{Sec:lkal}. We show that if $A$ is infinite and $\kal{A}$ is continuous, then $A$ contains an infinite subset $B$ such that the linear Kalton map $\lkal{B}$ is a topologically isomorphic embedding; see Corollary \ref{continuous:contain:isomorphic:subsets}.
We use this result in  the final Section \ref{Sec:final}, where we provide a characterization of topological vector spaces that contain the topological vector space  $\R^{(\N)}$ (Theorem \ref{complete:TVS:not:TAP:iff:contains:R:to:N})  and the topological vector space $\R^\N$ (Theorem \ref{another:theorem}).  As corollaries  we obtain the result of Lipecki about metric vector spaces containing $\R^{(\N)}$ (\cite[Theorem 3]{Lip}) and results of Bessaga, Pelczynski and Rolewicz about complete metric vector spaces containing $\R^\N$ (\cite[Theorem 9 and Corollary]{BPR}).  Theorem \ref{another:theorem} also answers a question of Hu\v{s}ek posed in \cite{Husek2}; see Remark \ref{answer:to:a:question:of:M:Husek}.

\section{Modified topology of a topological group capturing topological properties of the Kalton map}

\begin{definition}
\label{def:modified:topology}
Let $G$ be a group, and let $A$ be its subset.
\begin{itemize}
\item[(i)] For every subset $W$ of $G$ containing $0$ and each finite set $F\subseteq A$, let 
\begin{equation}
\label{eq:WAF}
\nbh{W}{A}{F}=\sum_{a\in F} (\hull{a} \cap W)+ \hull{A\setminus F}.
\end{equation}
\item[(ii)] For a group topology $\tau$ on $G$, we use $\t{A}{\tau}$  to denote  the group topology on $G$ having as its base of neighbourhoods of $0$ the family\begin{equation}
\{\nbh{W}{A}{F}: 0\in W\in \tau, F\subseteq A\mbox{ is finite}\},
\end{equation}
and we shall call $\t{A}{\tau}$ the {\em $A$-modification of $\tau$\/}.
\item[(iii)] For brevity, the $A$-modification of the discrete topology  $\delta_G$ on $G$ will be denoted by $\td{A}$ (instead of the longer notation $\t{A}{\delta_G}$).
\end{itemize}
\end{definition}

\begin{remark}
\label{linear:remark}
It follows easily from \eqref{eq:WAF} that $\nbh{\{0\}}{A}{F}=\hull{A\setminus F}$. Combing this with items (ii) and (iii) of Definition \ref{def:modified:topology}, we conclude that the family
$$
\{\hull{A\setminus F}: F\subseteq A, F \mbox{ is finite}\}
$$
forms a base of neighbourhoods of $0$ for the topology $\td{A}$. In particular, $\td{A}$ is a {\em linear\/} topology on $G$; that is, $\td{A}$ has a base at $0$ consisting of subgroups of $G$.
 \end{remark}

The topology $\t{A}{\tau}$ need not be Hausdorff in general, even when $\tau$ itself is Hausdorff.

For two topologies $\tau$, $\tau'$ on a set $X$, $\tau\le\tau'$ means that every $\tau$-open subset of $X$ is also $\tau'$-open.

\begin{remark}\label{observation:about:t:A}
Let $G$ be a group and $\tau$ be a group topology on $G$.
\begin{itemize}
 \item[(a)] If $\tau'$ is a group topology on $G$ with $\tau\le\tau'$, then $ \t{A}{\tau}\le \t{A}{\tau'}$; in particular, $\t{A}{\tau} \leq \td{A}$.
 \item[(b)] $\langle A \rangle$ is an open subgroup of $(G,\t{A}{\tau})$ and thus of $(G,\td{A})$ as well.
 \item[(c)] $\t{A}{\td{A}}=\td{A}$.
\end{itemize}
\end{remark}

\begin{proposition}
\label{Matsuyama:split}
For a subset $A$ of a topological group $(G,\tau)$, the following conditions are equivalent:
\begin{itemize}
 \item[(i)] $\tau \leq \t{A}{\tau}$; 
 \item[(ii)] $\tau \leq \td{A}$;
 \item[(iii)] $\tau\restriction_{\langle A\rangle}\le\t{A}{\tau}\restriction_{\langle A\rangle}$.
\end{itemize}
\end{proposition}
\begin{proof}
(i)$\to$(ii) follows from Remark
\ref{observation:about:t:A}~(a).

(ii)$\to$(i) Pick an arbitrary $U\in\tau$ with  $0\in U$. Since $\tau$ is a group topology, we can fix a $\tau$-neighbourhood $V$ of $0$ such that ${V}+ {V}\subseteq U$. 
Since $\tau\le \td{A}$ by (ii), $V\in \td{A}$. By Remark \ref{linear:remark}, there exists a finite set $F\subseteq A$ such that 
$\hull{A\setminus F}\subseteq {V}$. If $F=\emptyset$, we let $W=V$. Otherwise, we pick a $\tau$-neighbourhood $W$ of zero of $G$ such that  
$W+W+\ldots+W\subseteq V$, where $|F|$-many $W$'s are taken in the sum. Then $\sum_{a\in F} (\hull{a} \cap W)\subseteq V$, so that
$$
\nbh{W}{A}{F}= \sum_{a\in F} (\hull{a} \cap W)+ \hull{A\setminus F}  \subseteq  V+{V}\subseteq U.
$$
Since $\nbh{W}{A}{F}\in \t{A}{\tau}$, this implies $U\in \t{A}{\tau}$.

(i)$\to$(iii) is trivial. 

(iii)$\to$(i) Let $U$ be a $\tau$-open subset of $G$ containing $0$. Then $V=U\cap \hull{A}\in \tau\restriction_{\langle A\rangle}$, which implies $V\in \t{A}{\tau}\restriction_{\langle A\rangle}$ by (iii). Since $\hull{A}\in \t{A}{\tau}$ by Remark \ref{observation:about:t:A}~(b), this implies that $V\in \t{A}{\tau}$. Since $0\in V\subseteq U$, we conclude that $U\in \t{A}{\tau}$.
\end{proof}

The importance of the $A$-modification of a group topology to the topic of this paper becomes clear from the next proposition and its corollary below.

\begin{proposition}
\label{when:is:kalton:continuous}
Let $A$ be a subset of a topological group $(G,\tau)$ such that $0\not\in A$. Then:
\begin{itemize}
 \item[(i)]  the Kalton map $\kal{A}$ is continuous if and only if $\tau\le\t{A}{\tau}$, and
 \item[(ii)] the Kalton map $\kal{A}$ is an open map onto $\hull{A}$ if and only if $\t{A}{\tau}\restriction_{\hull{A}}\le\tau\restriction_{\hull{A}}$. \end{itemize}
\end{proposition}

\begin{proof}
Note that the topology of $S_A$ has the family 
$$
\{\nbh{W}{A}{F}^*: 0\in W\in \tau, F\subseteq A\mbox{ is finite}\}
$$
as its base at $0$, where 
$$
\nbh{W}{A}{F}^*=\left(\bigoplus_{{a}\in {F}}(\hull{{a}} \cap W)\right)\oplus
\left(\bigoplus_{{a}\in {A}\setminus {F}} \hull{{a}}\right).
$$

A straightforward check shows that $\kal{A}(\nbh{W}{A}{F}^*) =\nbh{W}{A}{F}$ for every finite set $F\subseteq A$ and every neighbourhood $W$ of $0$ in $(G,\tau)$, where $\nbh{W}{A}{F}$ are as defined in \eqref{eq:WAF}. The rest follows from Definition \ref{def:modified:topology}.
\end{proof}

\begin{corollary}
For a subset $A$ of a topological group $G$ such that $0\not\in A$, the following conditions are equivalent:
\begin{itemize}
\item[(i)] the Kalton map $\kal{A}$ is a topologically isomorphic embedding;
\item[(ii)] the set $A$ is independent and $\t{A}{\tau}\restriction_{\langle A\rangle}=\tau\restriction_{\langle A\rangle}$, where $\tau$ is the topology of $G$.
\end{itemize}
\end{corollary}
\begin{proof}
The inequalities $\tau \leq \t{A}{\tau}$ and $\tau\restriction_{\langle A\rangle}\le\t{A}{\tau}\restriction_{\langle A\rangle}$ are equivalent by Proposition 
\ref{Matsuyama:split}.
Therefore, the conclusion of our corollary follows from Proposition \ref{when:is:kalton:continuous}.
\end{proof}

\section{Absolutely Cauchy summable sets and the continuity of the Kalton map}
\label{sec:abs:cauch:summable:sets}

\begin{definition}\label{def:abs:summ:set}
We say that a subset $A$ of a topological group $G$ is {\em absolutely Cauchy summable\/} provided that for every neighbourhood $U$ of $0$ there exists a finite set $F\subseteq A$ such that $\hull{A\setminus F}\subseteq U$.
\end{definition}

\begin{remark}
\label{remark:on:abs:C:summability}
\begin{itemize}
\item[(i)]
Every finite subset of a topological group is absolutely Cauchy summable.
\item[(ii)] A subset $A$ of a topological group $G$ is absolutely Cauchy summable in $G$ if and only if $A$ is absolutely Cauchy summable in the subgroup $\hull{A}$ of $G$. 
\end{itemize}
\end{remark}

\begin{remark}
\label{splitted:remark}
Let $A$ be a subset of a group $G$. Clearly, $A$ is absolutely Cauchy summable in $(G, \td{A})$. Therefore, {\em $A$ is absolutely Cauchy summable in $(G, \tau)$
for every topology $\tau$ on $G$ satisfying $\tau\le \td{A}$\/}.
\end{remark}

A typical example of an absolutely Cauchy summable set appears in direct sums.
\begin{lemma}
\label{Cauchy:summable:in:direct:sum}
Let $\{H_a:a\in A\}$ be a family of topological groups and $H=\bigoplus_{a\in A} H_a$ be its direct sum.
If $x_a\in H_a$ for every $a\in A$, then the set $X=\{x_a:a\in A\}$ is absolutely Cauchy summable in $H$.
\end{lemma}

\begin{proof} Let $U$ be a neighbourhood of $0$ in $H$. By the definition of the Tychonoff product topology,
there exists a finite set $F\subseteq A$ such that  $\bigoplus_{a\in A\setminus F} H_a\subseteq U$. Clearly, $Y=\{x_a:a\in F\}$ is a finite subset of $X$ such that
$\hull{X\setminus Y}\subseteq \bigoplus_{a\in A\setminus F} H_a$. This gives $\hull{X\setminus Y}\subseteq U$. Therefore, $X$ is absolutely Cauchy summable in $H$ by Definition 
\ref{def:abs:summ:set}.
\end{proof}

\begin{theorem}\label{Proposition:Udine} For a subset $A$ of a  topological group $G$ such that $0\not\in A$, the following conditions are equivalent: 
\begin{itemize} 
 \item[(i)] $A$ is absolutely Cauchy summable in $G$;
 \item[(ii)] 
$\kal{A}: S_A \to G$ is continuous.
\end{itemize}
\end{theorem}

\begin{proof} 
Let $\tau$ be the topology of $G$.

(i)$\to$(ii)
Let $U$ be an arbitrary $\tau$-neighbourhood of $0$ in $G$. Since $A$ is absolutely Cauchy summable in $(G,\tau)$, there exists a finite set $F\subseteq A$ such that $\hull{A\setminus F}\subseteq U$.
Since $\hull{A\setminus F}$ is a $\td{A}$-neighbourhood of $0$ in $G$ by Remark
\ref{linear:remark}, we conclude that $U\in \td{A}$. This establishes the inclusion $\tau \leq \td{A}$. Applying Proposition \ref{Matsuyama:split}, we get the inclusion $\tau \leq \t{A}{\tau}$. Combining this with Proposition \ref{when:is:kalton:continuous}~(i), we obtain the continuity of the Kalton map $\kal{A}: S_A \to G$.

(ii)$\to$(i) By (ii) and Proposition \ref{when:is:kalton:continuous}~(i), the inequality 
$\tau\le\t{A}{\tau}$ holds. By Remark \ref{observation:about:t:A}~(a), $\t{A}{\tau}\le\td{A}$. Therefore, $A$ is absolutely Cauchy summable in $(G,\tau)$ by Remark \ref{splitted:remark}.
\end{proof}

\section{Topologically independent sets}
\label{sec:top:indep:sets}

Recall that a subset $A$ of non-zero elements of  a group $G$ is {\em independent\/} provided that for every finite subset $F\subseteq A$ and every indexed set $\{z_a:a\in F\}$ of integers the equality $\sum_{a\in F}z_aa=0$ implies that $z_aa=0$ for all $a\in F$. Our next definition is a topological analogue of this classical notion.

\begin{definition}
\label{def:topological:independence}
A subset $A$ of non-zero elements of  a topological group $G$ is called {\em topologically independent\/} provided that for every neighbourhood $W$ of $0$ there exists neighbourhood $U$ of $0$ such that for every finite subset $F\subseteq A$ and every indexed set $\{z_a:a\in F\}$ of integers the inclusion $\sum_{a\in F}z_aa\in U$ implies that $z_aa\in W$ for all $a\in F$. We will call this $U$ a {\em $W$-witness\/} of the topological independence of $A$.
\end{definition}

Informally speaking, $A$ is topologically independent if for every neighbourhood $W$ of $0$ there exists a neighbourhood $U$ of $0$ such that, whenever $\sum_{a\in F}z_aa$ is $U$-close to zero, then all $z_aa$ are $W$-close to zero. Thus, in the definition of a topologically independent set, the algebraic ``equality to zero" from the definition of an independent set is replaced by the topological notion of ``being close to zero". 

It is clear that a subset of a discrete topological group is topologically independent precisely when it is independent, so one can view topological independence as a generalization of the classical notion.

The following lemma shows that the notion of topological independence is also a strengthening of the classical notion.

\begin{lemma}\label{lemma:top:ind:is:ind}
Every topologically independent set is independent.
\end{lemma}   

\begin{proof}
Assume that $F$ is a finite subset of a topologically independent set $A$ in a topological group $G$ and $\{z_a:a\in F\}$ is an indexed set of integers such that $\sum_{a\in F}z_aa=0$. Fix $a\in F$. Let $W$ be an arbitrary neighbourhood of $0$. Since $A$ is topologically independent, we can find $U$ which is a $W$-witness of this. Since $\sum_{a\in F}z_aa=0\in U$, it follows that  $z_aa\in W$. Since this holds for an arbitrary $W$ and $G$ is (silently assumed to be) Hausdorff, this implies $z_aa=0$. This means that $A$ is independent.
\end{proof}

The next lemma and example that follows it both highlight the difference between the topological independence and the classical (algebraic) independence.

\begin{lemma}\label{lemma:step:to:Lie}
For every $n\in \N$, each topologically independent subset of $\R^n$  has size at most $n$.
\end{lemma}

\begin{proof}
Assume that $F$ is a finite topologically independent subset of $\R^n$ of size $n+1$. Since all cyclic subgroups of $\R^n$ are discrete, we can choose an open ball $W$ around $0$ such that 
\begin{equation}
\label{eq:5}
W\cap \hull{a}=\{0\}
\ 
\mbox{ for all }
\ 
a\in F.
\end{equation}
Since $F$ is topologically independent,  we can fix an open set $U$ containing $0$ which is a $W$-witness of the topological independence of $F$.
Since $F$ is topologically independent, it is independent by Lemma \ref{lemma:top:ind:is:ind}. Let $H$ be the closure of $\hull{F}$.  Since $F$ contains $n+1$-many elements, the closed subgroup $H$ of $\R^n$  is non-discrete; in fact, $H$ contains a line passing through $0$, see \cite[Chap.VII, Th.1 and Prop.3]{Bourbaki}. Therefore, since $U$ is an open set containing $0$, 
one can find some non-zero element $\sum_{a\in F}z_aa\in U$,  where $z_a\in\Z$ for $a\in F$.  Since $U$ is a $W$-witness, this implies that $z_aa\in W$
for all $a\in F$. Combining this with \eqref{eq:5}, we conclude that $z_aa=0$ for all $a\in F$. This contradicts the fact that $\sum_{a\in F}z_aa\not=0$.
\end{proof}

\begin{example}\label{discrete:top:indep:iff:indep} 
While the topological group $\R$ contains an independent subset of size $\cont$, every topologically independent subset of $\R$ has size at most $1$, by Lemma \ref{lemma:step:to:Lie}.
\end{example}

\begin{remark}\label{lemma:obvious}
A subset $A$ of non-zero elements of a topological group $G$ is topologically independent if and only if $A$ is topologically independent in the subgroup $\hull{A}$ of $G$.
\end{remark}

The next lemma provides a typical example of a topologically independent set.
\begin{lemma}
\label{topologcally:independent:in:direct:sum}
Let $\{H_a:a\in A\}$ be a family of topological groups and $H=\bigoplus_{a\in A} H_a$ be its direct sum.
If $x_a\in H_a\setminus \{0\}$ for every $a\in A$, 
then the set $X=\{x_a:a\in A\}$ is topologically independent in $H$.
\end{lemma}

\begin{proof}
Fix a neighbourhood $W$ of $0$ in $H$. By the definition of the Tychonoff product topology, there exist a finite set $F\subseteq A$ 
and an open neighbourhood $V_a$ of $0$ in $H_a$ such that $U\subseteq W$, where
$U=\bigoplus_{a\in F} V_a\oplus\bigoplus_{a\in A\setminus F}H_a$.
Observe that $U$ is a $U$-witness of the topological independence of $X$. As $U\subseteq W$, it is a $W$-witness as well.
\end{proof}

Our next lemma shows that topological independence of a set $A$ assures that the Kalton map is open.

\begin{proposition}\label{Proposition:Matsuyama}
Let $G$ be a topological group, $A\subseteq G\setminus\{0\}$
and let $\kal{A}:S_A\to G$ be the associated Kalton homomorphism.
\begin{itemize}
 \item[(i)] If $A$ is  topologically independent in $G$, then $\kal{A}$ is an open map onto $\hull{A}$.
 \item[(ii)] If $\kal{A}$ is a topologically isomorphic embedding, then $A$ is topologically independent in $G$.
\end{itemize}
\end{proposition}

\begin{proof}
(i)
Let $\tau$ be the topology of $G$.
By Proposition \ref{when:is:kalton:continuous}(ii), it suffices to check that $\t{A}{\tau}\restriction_{\hull{A}}\le\tau\restriction_{\hull{A}}$.
Fix a $\tau$-neighbourhood $W$ of $0$ and a finite set $F\subseteq  A$. Since $A$ is topologically independent, we can find a $\tau$-neighbourhood $U$ of $0$ which is a $W$-witness of topological independence of $A$. Then $U\cap\hull{A}\subseteq\nbh{W}{A}{F}$. Combined with Definition \ref{def:modified:topology}, this establishes the inequality $\t{A}{\tau}\restriction_{\hull{A}}\le\tau\restriction_{\hull{A}}$.

(ii) By our assumption, $\kal{A}:S_A\to G$ is a topologically isomorphic embedding.
For every $a\in A$, let $x_a=\kal{A}^{-1}(a)$. Then $x_a$ is a non-zero element of the summand $\hull{a}$ of $S_A$. By Lemma \ref{topologcally:independent:in:direct:sum}, the set $X=\{x_a:a\in A\}$ is topologically independent in $S_A$. Since $\kal{A}:S_A\to G$ is a topologically isomorphic embedding, 
$\kal{A}:S_A\to \kal{A}(S_A)=\hull{A}$ is a topological isomorphism, so 
$A=\kal{A}(X)$ is topologically independent in $\hull{A}$. By Remark \ref{lemma:obvious}, $A$ is topologically independent also in $G$.
\end{proof}

\begin{proposition}\label{direct:sum:top:indep} \label{top:indep:for:finite}
For a finite subset $A$ of a topological group $G$, the following conditions are equivalent:
\begin{itemize}
 \item[(i)] the Kalton map $\kal{A}: S_A\to G$ is a topologically isomorphic embedding;
 \item[(ii)] $A$ is topologically independent in $G$.
\end{itemize} 
\end{proposition}
\begin{proof}
The implication (i)$\to$(ii) was proved in Proposition \ref{Proposition:Matsuyama}~(ii).

(ii)$\to$(i) By Lemma \ref{lemma:top:ind:is:ind}, $A$ is independent. Hence the Kalton map $\kal{A}$ is an (algebraical) isomorphism.
Since $A$ is finite, it is absolutely Cauchy summable by Remark \ref{remark:on:abs:C:summability}~(i).
Applying Theorem \ref{Proposition:Udine}, we conclude that 
$\kal{A}$ is continuous. By Proposition \ref{Proposition:Matsuyama}~(i), $\kal{A}$ is an open map onto $\hull{A}$. This finishes the proof of (i).
\end{proof}

In the rest of this section we discuss topologically independent sets arising in functional analysis.
Recall that a sequence $\{e_i:i\in\N\}$ in a normed vector space $V$ is called a {\em Schauder basis} provided that for every $v\in V$ there exists unique sequence $\{s_i:i\in\N\}$ of scalars such that $v=\sum_{n=0}^\infty s_ne_n$, where the convergence is taken with respect to the norm; that is, $\lim_{n\to\infty}||v-\sum_{i=0}^n s_i e_i ||= 0.$  

\begin{proposition}\label{schauder:is:top:ind}
Every Schauder basis in a Banach space is topologically independent.
\end{proposition}

\begin{proof}
Let $E=\{e_i:i\in\N\}$ be a Schauder basis in a Banach space $B$. For $x=\sum_{i=0}^\infty a_ie_i\in B$ and $n\in\N$ put $P_n(x)=\sum_{i=0}^na_ie_i$. Then $P_n:B\to B$ is a linear projection for each $n\in\N$. Furthermore, by a theorem of Banach, these projections are uniformly bounded; see, for instance, \cite[Theorem 237]{HHZ}. That is, there exists $C\in\R$ such that 
$$
||P_n(x)||\le C||x|| \mbox{ for every $n\in\N$ and $x\in B$}.
$$ 
Put $P_{-1}(x)=0$ for each $x\in B$. Define $\pi_n=P_n-P_{n-1}$. Then for all $x\in B$ and $n\in\N$ we have

\begin{equation}\label{eq:schauder}
||\pi_n(x)||=||P_n(x)-P_{n-1}(x)||\le||P_n(x)||+||P_{n-1}(x)||\le 2C||x||. 
\end{equation} 

Take a neighbourhood $W$ of $0$. Then there exists $\varepsilon>0$ such that the open ball $B(0;\varepsilon)$
with the center at $0$ and  diameter $\varepsilon$ is a subset of $W$. We claim that the open ball $B(0;\frac{\varepsilon}{2C})$ is a $W$-witness of the topological independence of $E$.
Indeed, if $x=\sum_{i=0}^na_ie_i\in B(0;\frac{\varepsilon}{2C})$,  then 
$$
||a_ie_i||=||\pi_i(x)||\le 2C||x||<2C \frac{\varepsilon}{2C}=\varepsilon
$$ 
by \eqref{eq:schauder}. Thus $a_ie_i\in B(0;\varepsilon)\subseteq W$  for all $i=0,\ldots,n$. 
\end{proof}

\begin{remark}\label{rem:on:Schauder}
A Banach space $B$ with an infinite Schauder basis $A$ never contains a topologically isomorphic copy of $S_A$ (as it contains no infinite direct sums at all). In view of Proposition \ref{schauder:is:top:ind},  this shows that 
(i) and (ii) of Proposition \ref{direct:sum:top:indep} are not equivalent in general.
\end{remark}

Recall that a subset $A$ of a vector space is {\em linearly
independent} if for every finite set $B\subseteq A$ and each set
$\{r_b:b\in B\}$ of real numbers the equality $\sum_{b\in B}r_bb=0$
implies $r_b=0$ for all $b\in B$. Equivalently, $A$ is linearly
independent if every finite set $B\subseteq A$ generates a
$|B|$-dimensional Euclidean space.

\begin{proposition}\label{prop:step:to:Lie}
Let $A$ be a subset of a topological vector space.
\begin{itemize}
\item[(i)] If $A$ is topologically independent, then it is linearly independent.
\item[(ii)] If $A$ is finite and linearly independent, then it is topologically independent.
\end{itemize}
 \end{proposition}

\begin{proof}
(i)
It is enough to prove that every non-empty finite subset $B$ of $A$ is linearly independent. Let $B\subseteq A$ be 
finite and non-empty. Then $\rhull{B}=\R^n$ is an $n$-dimensional Euclidean vector space for a suitable $n\in\N$. Since $B$ generates $\R^n$, one has $n\le |B|$.
Being a subset of a topologically independent set $A$, $B$ itself is topologically independent.  By Lemma \ref{lemma:step:to:Lie}, the converse inequality $|B|\le n$ also holds. Therefore, $|B|=n$. Since $\rhull{B}=\R^n$, $B$ must be the basis for $\R^n$. In particular, $B$ is linearly independent.

(ii)
It is a well-known and simple fact that  $\rhull{A}=\bigoplus_{a\in A}\rhull{a}=\R^{|A|}$. Consequently, the set $A$ is topologically independent by 
Lemma \ref{topologcally:independent:in:direct:sum} and Remark \ref{lemma:obvious}. 
\end{proof}

\begin{remark}
\label{remark:4.12}
In \cite{ES}, the following stronger version of a linear independence of a countably infinite subset $A=\{a_i:i\in\N\}$ of a topological vector space is introduced:

\begin{equation}\label{eq:dalsi:rovnice}
\mbox{If } \sum_{i=0}^\infty r_ia_i=0 \mbox{ for some real sequence } \{r_i:i\in\N\} \mbox{, then } r_i=0 \mbox{ for all } i\in\N.
\end{equation}

This notion is studied under the names $\omega$-independent in \cite{Lip},  $\omega$-linearly independent in \cite{Singer} and linearly topologically independent in \cite{LL,Lip1}.
\end{remark}

\begin{remark}
If in Definition \ref{def:topological:independence} we assume $G$ to be a topological vector space and  replace the set of integers $\{z_a:a\in F\}$ by the set $\{r_a:a\in F\}$ of real numbers, then we obtain another generalization of linear independence for topological vector spaces which is stronger than the topological independence from Definition \ref{def:topological:independence}.
It is an interesting question how these two notions are related. However, this question is beyond the scope of this paper.
\end{remark}

\section{Direct sums in topological groups}

\begin{theorem}\label{Kalton:top:iso}
For a subset $A$ of a topological group $G$, the following conditions are equivalent:
\begin{itemize}
 \item[(i)] $A$ is both topologically independent and absolutely Cauchy summable in $G$;
 \item[(ii)] the Kalton map $\kal{A}$ is a topologically isomorphic embedding.
\end{itemize}
\end{theorem}

\begin{proof}
(i)$\to$(ii)
Since $A$ is topologically independent, it is independent by Lemma \ref{lemma:top:ind:is:ind}, and so the Kalton map $\kal{A}:S_A\to G$ is a monomorphism. By Proposition \ref{Proposition:Matsuyama}~(i), $\kal{A}$ is an open map onto its image $\kal{A}(S_A)=\hull{A}$. Since $A$ is absolutely Cauchy summable in $G$, the map $\kal{A}$
is continuous by Theorem \ref{Proposition:Udine}. This finishes the proof of (ii).

(ii)$\to$(i) Proposition \ref{Proposition:Matsuyama}~(ii) implies that $A$ is topologically independent in $G$, while Theorem \ref{Proposition:Udine} yields that $A$ is absolutely Cauchy summable.
\end{proof}

From this theorem and Fact \ref{fact}, we obtain the following corollary. 
\begin{corollary}
\label{sum:corollary}
For a topological group $G$ and a cardinal $\kappa$, the following conditions are equivalent:
\begin{itemize}
 \item[(i)] $G$ contains a subgroup topologically isomorphic to a direct sum of $\kappa$-many non-trivial groups;
 \item[(ii)] $G$ contains a topologically independent absolutely Cauchy summable set of size $\kappa$.
\end{itemize}
\end{corollary}

\begin{remark}\label{remark:Zp} 
Let $\Z_p$ denote the (compact metric) group of $p$-adic integers. It is known that $\Z_p$ is a linear group; that is, it has a basis at $0$ consisting of its clopen subgroups; see, for instance,  \cite{DPS}.

(i) It is shown in \cite{DSS_arxiv} that {\em every infinite null sequence $A$ in $\Z_p$ is absolutely Cauchy summable in $\Z_p$\/}.

(ii) It is known that {\em $\Z_p$ does not contain any sum of two non-trivial topological groups\/}.

(iii) It is a simple fact that {\em $\Z_p$ contains an infinite independent null sequence\/}.

(iv) It follows from (i), (ii) and (iii) that one cannot replace ``topologically independent'' in item (ii) in Corollary \ref{sum:corollary} by the weaker condition ``independent'', even when $G$ is a compact metric linear group.

(v) Let $A$ be an infinite null sequence in $\Z_p$.
It follows from (i), (ii) and Corollary \ref{sum:corollary} that $A$ is an absolutely Cauchy summable set such that the only topologically independent subsets of $A$ are singletons.

(vi)
{\em There exists an infinite absolutely Cauchy summable subset $B$ of a compact metric group such that 
the Kalton map $\kal{B}$ is continuous but not open.} Indeed, let $B$ be an infinite independent sequence in $\Z_p$
as in
(iii). By (i), $B$ is absolutely Cauchy summable, so $\kal{B}$ is continuous by Theorem \ref{Proposition:Udine}. Since $B$ is independent, $\kal{B}$ is injective. If it were also open, it would be a topologically isomorphic embedding in contradiction with 
item
(ii).
\end{remark}

Item 
(v) of the above remark shows that an infinite absolutely Cauchy summable set in a compact metric linear group need not contain any infinite topologically independent subset. Our next theorem shows that, under an additional condition imposed on an
infinite absolutely Cauchy summable set, it does contain an infinite topologically independent subset.

\begin{theorem}\label{Gn:discrete:gives:injective:heomeo}
Let $A$ be an infinite absolutely Cauchy summable set in a topological group $G$ such that $\hull{a}$
 is discrete for every $a\in A$. Then $A$ contains an infinite topologically independent subset.
\end{theorem}
\begin{proof}
  We will build the topologically independent faithfully indexed subset $B=\{a_{n}:n\in\N\}$ of $A$ by induction on $n\in\N$.  At each step we choose an element $a_n\in A$, a finite set $F_n\subseteq A$ and an open symmetric neighbourhood $U_n$ of $0$ satisfying the conditions (i$_n$)--(iv$_n$) listed below.
First, we define $F_{-1}=\emptyset$ and $U_{-1}=G$. 

\begin{itemize}
\item[(i$_n$)] $F_{n-1}\subseteq F_n$ and $a_n\in F_n\setminus F_{n-1}$,
\item[(ii$_n$)] $(U_n+U_n)\cap \hull{a_{n}}=\{0\}$,
\item[(iii$_n$)] $\hull{A\setminus F_{n}}\subseteq U_{n}$,
\item[(iv$_n$)] $U_n\subseteq U_{n-1}$.
\end{itemize}

Suppose that for $n\in\N$ a finite set $F_{n-1}\subseteq A$ and an  open symmetric neighbourhood $U_{n-1}$  of $0$ have already been selected. Let us define $a_{n}\in A$, a finite set $F_{n}\subseteq A$ and an open symmetric neighbourhood $U_{n}$ of $0$ satisfying conditions (i$_{n}$)--(iv$_{n}$). 

Since $A$ is infinite and $F_{n-1}$ is finite, we can choose $a_n\in A\setminus F_{n-1}$. Since $\hull{a_{n}}$ is discrete by our assumption, we can fix a symmetric neighbourhood $U_{n}$ of $0$ satisfying (ii$_{n}$). By choosing a smaller $U_n$ if necessary, we may also assume that it satisfies the condition (iv$_{n}$) as well. Since 
$A$ is absolutely Cauchy summable, there exists a finite set $E$ such that $\hull{A\setminus E}\subseteq U_{n}$. Clearly, $F_n=\{a_n\}\cup F_{n-1}\cup E$ is a finite subset of $A$ satisfying 
(i$_n$). Since $E\subseteq F_n$, we have $\hull{A\setminus F_{n}}\subseteq\hull{A\setminus E}\subseteq U_{n}$; that is,
(iii$_n$) holds as well. This finishes our inductive construction.

Since (i$_n$) holds for every $n\in \N$, we conclude that 
\begin{equation}
\label{eq:9}
B=\{a_n:n\in\N\}\subseteq \bigcup_{n\in\N}F_n\subseteq A
\end{equation}
and $a_n\not=a_m$ for $n,m\in \N$ and $m\not=n$.
In particular, $B$ is infinite.

Let us show that  $B$ is topologically independent.
Fix a neighbourhood $W$ of $0$. Since $A$ is absolutely Cauchy summable, so is its subset $B$. Therefore,
 $\hull{B\setminus S}\subseteq W$ for some finite set $S\subseteq B$. Since $\{F_n:n\in\N\}$ is an increasing sequence of sets by (i$_n$), \eqref{eq:9} allows us to find an $n\in\N$ such that $S\subseteq F_n$. Now
\begin{equation}\label{eq:B:minus:Fn:in:W}
\hull{B\setminus F_{n}}\subseteq \hull{B\setminus S}\subseteq W.
\end{equation}

We claim that $U_n$ is a $W$-witness of the topological independence of $B$. Indeed, take a finite set $F\subseteq B$ and an indexed set $\{z_a:a\in F\}$ of integers such that 
\begin{equation}\label{eq:c:in:Un}
c=\sum_{a\in F}z_aa\in U_n.
\end{equation}
Without loss of generality, we may assume that $z_aa\neq 0$ for all $a\in F$.

We are going to show that $z_aa\in W$ for each $a\in F$.
To achieve this, it suffices to check that $F\cap F_n=\emptyset$.
Indeed, assuming that this has already been proved, 
for every $a\in F$
we would have $a\in F\setminus F_n\subseteq B\setminus F_n$, so
$z_aa\in \hull{B\setminus F_{n}}\subseteq W$ by \eqref{eq:B:minus:Fn:in:W}. 

Therefore, we shall assume that  $F\cap F_n\not=\emptyset$ and derive a contradiction from it. Let $m\in\N$ be the minimal element with the property that $F\cap F_m\not=\emptyset$.
Since $F\cap F_n\not=\emptyset$ by our assumption, $m\le n$. Since  $F\subseteq B=\{a_n:n\in\N\}$ and 
(i$_k$) holds for every $k\in\N$, from the minimality of $m$ one concludes that $F\cap F_m=\{a_m\}$. Therefore, 
\begin{equation}
\label{new:eq}
c-z_{a_m} a_m\in\hull{F\setminus F_m}\subseteq \hull{A\setminus F_m}\subseteq U_m
\end{equation}
by \eqref{eq:c:in:Un} and (iii$_m$). 
Since (iv$_k$) holds for every $k\in\N$ and $m\le n$, it follows that  $U_n\subseteq U_m$. Combining this with \eqref{eq:c:in:Un}, we get 
$c\in U_m$. Since $U_m$ is symmetric, this and \eqref{new:eq} yield $z_{a_m} a_m\in U_m+U_m$. Recalling (ii$_m$), we get 
$z_{a_m} a_m=0$. On the other hand, since $a_m\in F$, we have $z_{a_m} a_m\not=0$ by our assumption.
This contradiction finishes the proof of the equality $F\cap F_n=\emptyset$. 
\end{proof}

\begin{corollary}
\label{precise:corollary}
Let $A$ be an infinite absolutely Cauchy summable set in a topological group $G$ such that $\hull{a}$
 is discrete for every $a\in A$. Then
$G$ contains a subgroup (topologically isomorphic to) $S_B$ for some infinite subset $B$ of $A$.
\end{corollary}

\begin{proof}
By Theorem \ref{Gn:discrete:gives:injective:heomeo}, $A$ contains an infinite topologically independent subset $B$. Since $A$ is absolutely Cauchy summable, 
so is its subset $B$. By Theorem \ref{Kalton:top:iso}, the Kalton map  $\kal{B}: S_B\to G$ is a topologically isomorphic embedding.
\end{proof}

\begin{corollary}\label{cor:thm:basic}
Let $G$ be a topological group such that each of its cyclic subgroups is discrete. Then the following statements are equivalent:
\begin{itemize}
\item[(i)]
$G$ contains an infinite absolutely Cauchy summable set;
\item[(ii)] $G$ contains a subgroup (topologically isomorphic to) $S_A$ for some infinite subset $A$ of $G\setminus\{0\}$.
\end{itemize}
\end{corollary}
\begin{proof}
The implication (i)~$\to$~(ii) follows from Corollary \ref{precise:corollary}.
The reverse implication (ii)~$\to$~(i) follows from the implication (ii)~$\to$~(i) of Theorem \ref{Kalton:top:iso}. 
\end{proof}

Remark \ref{remark:Zp} shows that the condition ``all cyclic subgroups are discrete'' cannot be weakened to ``all cyclic subgroups have linear topology'' in 
this corollary, as well as in Theorem \ref{Gn:discrete:gives:injective:heomeo} and Corollary \ref{precise:corollary}. 

\section{Absolutely summable sets versus absolutely Cauchy summable sets}
\label{sec:abs:summable:sets}
\label{sec:KALTON}

\begin{definition}
We say that a subset $A$ of a topological group $G$ is 
{\em absolutely summable} in $G$ provided that, for every family $\{z_a:a\in A\}$ of integers indexed by $A$, there exists $g\in G$ having the following property:
For every neighbourhood $U$ of $0$ one can find a finite set $F\subseteq A$ such that 
\begin{equation}\label{eq:def:sum} g-\sum_{a\in E}z_a a\in U \mbox{ for every finite } 
E\subseteq A \mbox{ containing } F; 
\end{equation}
that is, the indexed set $\{z_aa:a\in A\}$ is summable in the sense of Bourbaki (see \cite[Appendice II, D\'{e}finition 1]{Bourbaki}). In this case we write
\begin{equation}\label{eq:g:is:sum}
g=\sum_{a\in 
A}z_a a.
\end{equation}
\end{definition}

The following lemma provides a typical example of an absolutely summable set. We omit its straightforward proof.

\begin{lemma}
\label{summable:sets:in:direct:product}
Let $\{H_a:a\in A\}$ be a family of topological groups and $H=\prod_{a\in A} H_a$ be its direct product. If $x_a\in H_a$ for every $a\in A$, then the set $X=\{x_a:a\in A\}$ is absolutely summable in $H$.
\end{lemma}

\begin{lemma}\label{cont:hom:preserves:summ}
Let $G,H$ be topological groups, $A$ an absolutely  summable set in $G$ and $f:G\to H$ a continuous homomorphism. Then $f(A)$ is an absolutely summable set in $H$.
If, moreover, $f$ is injective on $A$, then  for every indexed set $\{z_a:a\in A\}$ of integers we have 
\begin{equation}\label{eq:f(g):is:sum}
f\left(\sum_{a\in A}z_aa\right)=\sum_{a\in A}z_af(a).
\end{equation}

\end{lemma}
\begin{proof}
We may assume that $f$ is injective on $A$ (otherwise we can replace $A$ by  $A'\subseteq A$ such that $f(A')=f(A)$ and $f$ is injective on $A'$).
Pick a family $\{z_a:a\in A\}$ of integers arbitrarily. Then there is $g\in G$ such that \eqref{eq:g:is:sum} holds. It suffices to show the equality \eqref{eq:f(g):is:sum} (since $f(A)$ is injective on $A$, we do not distinguish between the summing indexes of $A$ and $f(A)$). Pick a neighbourhood $V$ of $0_H$. Then there is finite $F\subseteq A$ such that \eqref{eq:def:sum} holds for $U=f^{-1}(V)$. Hence $$f(g)-\sum_{a\in E}z_af(a)=f\left(g-\sum_{a\in E}z_aa\right)\in f(U)=f(f^{-1}(V))\subseteq V$$ holds for every finite $E\subseteq A$ containing $F$. This yields \eqref{eq:f(g):is:sum}. Hence $f(A)$ is absolutely summable.
\end{proof}

\begin{proposition}
\label{exchange of quantifiers generalized}
Given a subset $A$ of elements of a topological group $G$, the following statements are equivalent:
\begin{itemize}
\item[(i)] For every neighbourhood $U$ of $0$ and every indexed set $\{z_a:a\in A\}$ of integers, there exists finite $F\subseteq A$ such that 
\begin{equation}\label{exchange:quantifiers}
\sum_{a\in E} z_aa\in U\  \mbox{ for every finite }\ E\subseteq A\setminus F.
\end{equation}

\item[(ii)]
For every neighbourhood $U$ of $0$  there exists finite $F\subseteq A$ such that for every indexed set  $\{z_a:a\in A\}$ of integers, we have (\ref{exchange:quantifiers}).

\item[(iii)]
$A$ is absolutely Cauchy summable.
\end{itemize}
\end{proposition}
\begin{proof}
Items (ii) and (iii) are obviously equivalent and (ii) trivially implies (i). 

If $A$ is finite, then (i) and (ii) are equivalent. Thus we may assume that $A$ is infinite. To show that (i) implies (ii) we will prove the contrapositive. Suppose that (ii) does not hold. That is, we can fix a neighbourhood $U$ of $0$ such that for every finite $F\subseteq A$ there exists finite $E\subseteq A\setminus F$ such that \eqref{exchange:quantifiers} does not hold. This allows us to pick inductively for each $n\in\N$  a finite $E_n\subseteq A$ and a corresponding indexed set $\{z_a:a\in E_n\}$ of integers such that the collection $\{E_n:n\in\N\}$ is pairwise distinct and \begin{equation}\label{eq:exchange:quantifiers}
\sum_{a\in E_n} z_aa\not\in U \mbox{ for all } n\in\N.
\end{equation} 
For $a\in A\setminus \bigcup_{n\in\N}E_n$ pick an integer $z_a$ arbitrarily. Then the neighbourhood $U$ and the indexed set $\{z_a:a\in A\}$ of integers witness that (i) does not hold. Indeed, if $F\subseteq A$ is an arbitrary finite set, then, since $E_n$'s are pairwise distinct, there is $k\in\N$ such that $E_k\subseteq A\setminus F$. Hence \eqref{exchange:quantifiers} does not hold by \eqref{eq:exchange:quantifiers}.
\end{proof}

\begin{proposition}\label{as:is:acs}
Every absolutely summable set in a topological group is absolutely Cauchy summable. The converse implication holds in complete groups.
\end{proposition}

\begin{proof} Let $A$ be an absolutely summable set in a topological group $G$. Fix a neighbourhood $U$ of $0$ and an indexed set $\{z_a:a\in A\}$ of integers. By Proposition \ref{exchange of quantifiers generalized} it suffices to find finite $F\subseteq A$ such that \eqref{exchange:quantifiers} holds. 

Let $s$ be the sum of $\{z_aa:a\in A\}$. Pick a symmetric neighbourhood $V$ of $0$ with the property that $V+V\subseteq U$.  By our assumption, there exists finite $F\subseteq A$ such that 

\begin{equation}\label{eq:as:acs}
s-\sum_{a\in E\cup F}z_aa=s-\sum_{a\in F}z_aa-\sum_{a\in E}z_aa\in V \mbox{ for each finite } E\subseteq A\setminus F.
\end{equation}

In particular, for $E=\emptyset$ we have
$$
s-\sum_{a\in F}z_aa\in V. 
$$ 
This together with \eqref{eq:as:acs} gives us
$$
\sum_{a\in E}z_aa\in V+V\subseteq U \mbox{ for every finite } E\subseteq A\setminus F,
$$
because $V$ was chosen symmetric. This gives us \eqref{exchange:quantifiers}.

To prove the second assertion, assume that $A$ is an absolutely Cauchy summable subset of a complete group $G$.
Then for every indexed set $\{z_a:a\in A\}$ of integers the net $\{\sum_{a\in F}z_aa:F\subseteq A, F$ is finite$\}$ is a Cauchy net in $G$, where the net order is given by the set inclusion.
Thus, this net converges to some element $g \in G$. A straightforward check shows that $g=\sum_{a\in A}z_aa$.
\end{proof}

\begin{question}
Let $G$ be a linear group. Assume that a subset of $G$ is absolutely Cauchy summable (if and) only if it is absolutely summable. Must $G$ be complete?
\end{question}

\begin{remark}\label{f:omgea:is:abs:sum}
Following \cite{DSS_arxiv} we say that a faithfully indexed sequence $\{a_i:i\in\N\}$ in a topological group $G$ is {\em $f_\omega$-(Cauchy) summable} provided that the sequence $\{\sum_{i=1}^nz_ia_i:n\in\N\}$ converges (is a Cauchy sequence) for every sequence $\{z_i:i\in\N\}$ of integers.  It is easy to see that a faithfully indexed sequence in a topological group is $f_\omega$-summable if and only if it is absolutely summable. Further, from Proposition \ref{exchange of quantifiers generalized} it follows that it is $f_\omega$-Cauchy summable if and only if it is absolutely Cauchy summable. For other versions of $f_\omega$-summability see \cite{DSS_arxiv,Spe}.
\end{remark}

\begin{remark}
\label{TAP:remark}
In \cite{SS} the notion of a TAP group was introduced and investigated.  (This property appeared earlier without a specific name in \cite{Husek1}; see ($\star$) on page 163 of \cite{Husek1}.) An equivalent definition from \cite{DSS_arxiv} says that a topological group is TAP if and only if it contains no $f_\omega$-summable sequences. Therefore, {\em a topological group $G$ is TAP if and only if every absolutely summable set in $G$ is finite\/}. Other variants of this property were studied in \cite{DT-private}.
\end{remark}

\begin{remark}
Inspired by the notion of a linearly independent set defined in Remark \ref{remark:4.12}, one can introduce a correspondent notion for topological groups.
Indeed, if we replace the real sequence 
$\{r_i:i\in\N\}$ in \eqref{eq:dalsi:rovnice} by a sequence $\{z_i:i\in\N\}$ of integers, then this notion makes sense also for a topological group, and thereby we obtain another generalization of independence of a set. (Here $\sum_{i=0}^\infty z_ia_i=0$ means that the net $\left\{\sum_{i\in I}z_ia_i: I\subseteq\N\mbox{ if finite}\right\}$ converges to $0$.)
\end{remark}

\section{Extension of the Kalton map and direct products}

We use $\overline{G}$ to denote the completion of an abelian topological group $G$.

Let $G$ be an abelian topological group.
As we have seen in Theorem \ref{Proposition:Udine}, the Kalton map $\kal{A}:S_A\to G$ is continuous if and only if $A$ is absolutely Cauchy summable. 
In such a case,
there exists a unique continuous extension 
$$
\skal{A}:\overline{S_A}\to \overline{G}
$$ 
of $\kal{A}$ which is also a group homomorphism.
Note that 
$$
P_A=\prod_{a\in A}\hull{a}\subseteq \prod_{a\in A}\overline{\hull{a}}=\overline{S_A},
$$
 so the restriction of $\skal{A}$ to $P_A$ is well-defined. With a slight abuse of notations, we shall denote this restriction also by $\skal{A}$.

In our next proposition we provide an exact condition when the continuous extension $\skal{A}: P_A\to\overline{G}$ of the Kalton map $\kal{A}$ over $P_A$ takes values in $G$ rather than in $\overline{G}$.  

\begin{proposition}\label{Proposition:Prague}
Let $A$ be a subset of a topological group $G$ such that $0\not\in A$. Then the following statements are equivalent:
\begin{itemize}
 \item[(i)] $A$ is  an absolutely summable set in $G$;
 \item[(ii)] the Kalton map $\kal{A}:S_A\to G$ is continuous and 
its continuous extension  $\skal{A}:P_A\to\overline{G}$ satisfies $\skal{A}(P_A)\subseteq G$.
\end{itemize}
Furthermore, if these equivalent conditions hold, then 
$\skal{A}(h)=\sum_{a\in A} \kal{A}(h_a)$
for every $h=\{h_a\}_{a\in A}\in P_A$.
\end{proposition}
\begin{proof}
For every $a\in A$ let $x_a$ be the unique element of the cyclic summand $\hull{a}$ of $S_A$ such that $\kal{A}(x_a)=a$.
Define $X=\{x_a:a\in A\}\subseteq S_A$.
Clearly, $A=\kal{A}(X)$ and $\kal{A}$ is injective on $X$.

(i)~$\to$~(ii)
Since $A$ is absolutely summable, it is absolutely Cauchy summable by Proposition \ref{as:is:acs}. By Theorem \ref{Proposition:Udine}, the Kalton map $\kal{A}$ is continuous. Therefore, there exists a unique continuous homomorphism $\skal{A}:P_A\to\overline{G}$ extending $\kal{A}$.

Let $h=\{h_a\}_{a\in A}\in P_A$ be arbitrary.  There exists an indexed set $\{z_a:a\in A\}$ of integers such that $h_a=z_a x_a$ for every $a\in A$. Then 
$
h=\sum_{a\in A}z_a x_a 
$.
Since $X$ is absolutely summable in $P_A$ by Lemma \ref{summable:sets:in:direct:product} and $\kal{A}$ (and thus, $\skal{A}$) is injective on $X$, 
from Lemma \ref{cont:hom:preserves:summ}
we conclude that
\begin{equation}
\label{eq:18}
\skal{A}(h)=\skal{A}\left(\sum_{a\in A}z_a x_a\right)=
\sum_{a\in A}z_a\skal{A}(x_a)
=
\sum_{a\in A}z_a\kal{A}(x_a)
=
\sum_{a\in A}z_aa.
\end{equation}
Since $A$ is absolutely summable by (i), $\sum_{a\in A}z_a a\in G$.  This shows that $\skal{A}(P_A)\subseteq G$.

(ii)~$\to$~(i)
The set $X$ is absolutely summable in $P_A$ by Lemma \ref{summable:sets:in:direct:product}.
Since $\skal{A}:P_A\to G$ is continuous by (ii),
Lemma \ref{cont:hom:preserves:summ} implies that the set $A=\skal{A}(X)$ is absolutely summable in $\skal{A}(P_A)$. Since $\skal{A}(P_A)\subseteq G$ by (ii), $A$ is absolutely summable in $G$ as well.

The final statement of the proposition follows from \eqref{eq:18}.
\end{proof}

\begin{theorem}\label{prod:in:group}\label{characterization:of:topological:isomorphism:of:extended:Kalton:maps}

For a subset $A$ of a topological group $G$ such that $0\not\in A$,
the following conditions are equivalent:
\begin{itemize}
 \item[(i)] the Kalton map $\kal{A}:S_A\to\overline{G}$ is continuous and its continuous extension  $\skal{A}:P_A\to\overline{G}$ is  a topologically isomorphic embedding satisfying $\skal{A}(P_A)\subseteq G$;
 \item[(ii)] the set $A$ is both absolutely summable and topologically independent in $G$.
\end{itemize}
\end{theorem}
\begin{proof}
(i)$\to$(ii) By our assumption and the implication (ii)~$\to$~(i) of Proposition \ref{Proposition:Prague}, $A$ is absolutely summable in $G$. Since $\kal{A}$ is a topologically isomorphic embedding, $A$ is topologically independent in $G$ by Proposition \ref{Proposition:Matsuyama} (ii).

(ii)$\to$(i)  
Since $A$ is absolutely summable by our assumption, it is Cauchy summable in $G$ by Proposition \ref{as:is:acs}.
Since $A$ is also topologically independent in $G$, applying the implication (i)~$\to$~(ii) of Theorem 
\ref{Kalton:top:iso}, 
we conclude that the Kalton map $\kal{A}: S_A\to G$ is a topologically isomorphic embedding. By the uniqueness of completions of topological groups, its continuous extension 
$\skal{A}: \overline{S_A}\to \overline{G}$ is also a topologically isomorphic embedding. Since $S_A\subseteq P_A\subseteq\overline{S_A}$, its restriction
$\skal{A}:P_A\to\overline{G}$  to $P_A$ is a topologically isomorphic embedding as well. It remains to note that $\skal{A}(P_A)\subseteq G$ by Proposition \ref{Proposition:Prague}.
\end{proof}

\begin{corollary}
For a topological group $G$ and a cardinal $\kappa$, the following conditions are equivalent:
\begin{itemize}
 \item[(i)] $G$ contains a subgroup topologically isomorphic to a direct product of $\kappa$-many non-trivial groups;
 \item[(ii)] $G$ contains a topologically independent absolutely summable set of size $\kappa$.
\end{itemize}
\end{corollary}

\begin{proof}
(i)~$\to$~(ii) Let $H$ be a subgroup of $G$ topologically isomorphic to a product $\prod_{a\in A} H_a$ of non-trivial topological groups $H_a$ such that $|A|=\kappa$. Select a non-zero element $x_a\in H_a$ for every $a\in A$ and let
$X=\{x_a:a\in A\}$. Clearly, $|X|=|A|=\kappa$ and $0\not\in X$.

By Lemma \ref{topologcally:independent:in:direct:sum}, $X$ is topologically independent in $\bigoplus_{a\in A} H_a$. Since this direct sum is a subgroup of $H$
which in turn is a subgroup of $G$, we conclude that $X$ is topologically independent in $G$. By Lemma \ref{summable:sets:in:direct:product}, $X$ is absolutely summable in $H$ (and thus, in $G$ as well).

(ii)~$\to$~(i) Let $A$ be a topologically independent absolutely summable subset of $G$ such that $|A|=\kappa$ and $0\not\in A$.
By the implication (ii)~$\to$~(i) of Theorem \ref{prod:in:group}, $G$ contains a subgroup topologically isomorphic to the direct product 
$P_A$. Since $0\not\in A$, each $\hull{a}$ for $a\in A$ is a non-trivial group.\end{proof}

\begin{corollary}\label{cor:one:more}
\label{new:products:corollary}
Let $G$ be a topological group such that each of its cyclic subgroups is discrete. Then the following statements are equivalent:
\begin{itemize}
\item[(i)]  $G$ contains an infinite absolutely summable set;
\item[(ii)]  $G$ contains a subgroup (topologically isomorphic to) $P_B$ for some infinite subset $B$ of $G$.
\end{itemize}
Furthermore, if $G$ is complete, then the following item is equivalent to the other two:
\begin{itemize} 
\item[(iii)] $G$ contains an infinite absolutely Cauchy summable set.
\end{itemize}
\end{corollary}

\begin{proof}
(i)$\to$(ii) Let $A$ be an infinite absolutely summable set in $G$. Then $A$ is absolutely Cauchy summable by Proposition \ref{as:is:acs}. 
By Theorem \ref{Gn:discrete:gives:injective:heomeo}, $A$ contains an infinite topologically independent subset $B$. Being a subset of an absolutely summable set $A$, $B$ is absolutely summable as well.
Now item (ii) follows from the implication (ii)$\to$(i) of  Theorem \ref{prod:in:group}. 

The implication (ii)$\to$(i) follows from Lemma \ref{summable:sets:in:direct:product}.

Finally, the equivalence of items (i) and (iii) for a complete group $G$ follows from Proposition \ref{as:is:acs}.
\end{proof}

\section{Applications to metric NSS groups}
\label{Sec:inf:dir:sums:in:top:groups}

The first theorem in this section demonstrates that the notion of an absolutely Cauchy summable set can be used to characterize the NSS property in metric groups.

\begin{theorem}\label{metric:NSS:iff:no:abs:C:summable}
A metric group $G$ is NSS if and only if every absolutely Cauchy summable set in $G$ is finite. 
\end{theorem}

\begin{proof}
By \cite[Theorem 5.5]{DSS_arxiv}, a group $G$ is NSS if and only if $G$ contains no $f_\omega$-summable sequences. According to Remark \ref{f:omgea:is:abs:sum}, 
this occurs precisely when every absolutely Cauchy summable set in $G$ is finite.
\end{proof}

The reader may wish to compare Theorem \ref{metric:NSS:iff:no:abs:C:summable} with  \cite[Theorem 3.9]{DT-private}.

\begin{theorem}\label{thm:basic}
Let $G$ be a metric group such that each of its cyclic subgroups is discrete.
Then the following statements are equivalent:
\begin{itemize}
\item[(i)] $G$ contains a subgroup (topologically isomorphic to) $S_A$ for some infinite subset $A$ of $G$;
\item[(ii)] $G$ is not NSS.
\end{itemize}
\end{theorem}
\begin{proof}
Combine Corollary \ref{cor:thm:basic} with Theorem \ref{metric:NSS:iff:no:abs:C:summable}.
\end{proof}

\begin{corollary}\label{sum:of:Z:inside}
Let $G$ be a torsion-free metric group such that each of its cyclic subgroups is discrete. Then the following statements are equivalent:
\begin{itemize}
 \item[(i)] $G$ is not NSS;
 \item[(ii)] $G$ contains a topologically isomorphic copy of $\Z^{(\N)}$.
\end{itemize}
If $G$ is also complete, then the next condition is equivalent to the above two:
\begin{itemize}
\item[(iii)] $G$ contains a topologically isomorphic copy of $\Z^{\N}$.
\end{itemize}
\end{corollary}
\begin{proof}
The equivalence of items (i) and (ii) follows from Theorem \ref{thm:basic}, as each cyclic subgroup $\hull{a}$ of $G$ is a copy of $\Z$ with the discrete topology. The equivalence of items (i) and (iii) follows from Corollary \ref{new:products:corollary} and Theorem \ref{metric:NSS:iff:no:abs:C:summable}.
\end{proof}

Remark \ref{remark:Zp} shows that the condition ``all cyclic subgroups are discrete'' cannot be weakened to ``all cyclic subgroups have linear topology'' in 
Theorem \ref{thm:basic}
 and Corollary \ref{sum:of:Z:inside}. 

In \cite[Section 6]{DT-private} a subset $U$ of a topological group $G$ is called {\em root invariant} if $\{y:ny=x \mbox{ for some } n\in\N\}\subseteq U$ whenever $x\in U$. The topological group $G$ is then called  {\em locally root invariant} provided that its topology  has a local base at $0$ consisting of root invariant sets. Finally, $G$ belongs to the class {\em MMP} provided that it is locally root invariant and it is metrizable by a translation-invariant metric $d$ with the property that 
\begin{equation}\label{eq:Vajas's}
d(0,g)\le d(0,ng) \mbox{ for all } g\in G\mbox{ and } n\in\N.
\end{equation}
Clearly, locally root invariant groups are torsion free, while condition \eqref{eq:Vajas's} implies, that $\hull{g}$ is discrete for every $g\in G$. Therefore,  Corollary \ref{sum:of:Z:inside} generalizes the following result from  \cite{DT-private}.

\begin{corollary}
Let $G$ be an MMP group.
\begin{itemize}
\item[(i)]
\cite[equivalence (i)$\leftrightarrow$(iii) of Theorem 6.9]{DT-private}
$G$ contains a topologically isomorphic copy of $\Z^{(\N)}$ if and only if $G$ is not NSS.
\item[(ii)]\cite[equivalence (i)$\leftrightarrow$(v) of Theorem 6.10]{DT-private} If $G$ is also complete then $G$ contains a topologically isomorphic copy of $\Z^{\N}$ if and only if $G$ is not NSS.
\end{itemize}
\end{corollary}

From Corollary \ref{new:products:corollary} and Theorem \ref{metric:NSS:iff:no:abs:C:summable} one obtains the following result:

\begin{theorem}\label{compact:zero:dim:inside}
Let $G$ be an infinite complete metric torsion group. Then the  following statements are equivalent:
\begin{itemize}
\item[(i)] $G$ is not NSS;
\item[(ii)] $G$ contains a subgroup topologically isomorphic to an infinite product of finite non-trivial groups.
\end{itemize} 
In particular, if $G$ is not NSS, then $G$ contains an infinite compact zero-dimensional subgroup.
\end{theorem}

\begin{remark}
All NSS groups are TAP; see Remark \ref{TAP:remark} for the definition of a TAP group. The implication NSS$\to$TAP was proved in \cite{SS}, although it was mentioned much earlier without a proof in \cite{Husek1}.
TAP groups contain no infinite products of non-trivial groups \cite{SS}. 
\end{remark}

\section{Continuity of finite-dimensional projections in topological vector spaces}

The next simple fact is well known; see \cite[1.3 in Chapter III]{Schaefer}.

\begin{fact}
\label{closed:kernels}
Let $E$ and $E'$ be topological vector spaces such that $E'$ is finite-dimensional. Then  
a linear map $f:E\to E'$ is continuous if and only if its kernel $\ker f=\{x\in E: f(x)=0\}$ is a closed subset of $E$.
\end{fact} 

\begin{proposition}
\label{R:Cauchy:summable}
Let $A$ be an absolutely Cauchy summable subset of a topological vector space $E$. Then:
\begin{itemize}
\item[(i)]
 for every neighbourhood $V$ of $0$ in $E$ there exists 
a finite set $F\subseteq A$ such that $\rhull{A\setminus F}\subseteq V$;
\item[(ii)] the subspace $\rhull{A}$ of $E$ is locally convex. 
\end{itemize}
\end{proposition}

\begin{proof} 
(i)
Recall that a subset $S$ of a vector space  is called {\em balanced} if $rS\subseteq S$ for every scalar $r$ with $|r|\le 1$. It is a folklore fact that every topological vector space has a base of its topology at $0$ consisting of balanced sets. Therefore, we can fix a closed balanced neighbourhood $U$ of $0$ such that $U\subseteq V$. Since $A$ is absolutely Cauchy summable in $E$, there exists a finite set $F \subseteq A$ such that $\hull{A\setminus F}\subseteq U$. Since $U$ is balanced, it contains all lines connecting $0$ with points from $\hull{A\setminus F}$.  Since the union of these lines is dense in $\rhull{A\setminus F}$ and $U$ is closed, $\rhull{A\setminus F}\subseteq U\subseteq V$.

(ii) 
Define $L=\rhull{A}$. Let $W$ be a neighbourhood of $0$ in $L$ and let 
$V$ be a  closed neighbourhood of $0$ in $L$ such that $V + V \subseteq W$. 
Apply item (i) to $V$ and $L$ (taken as $E$) to get a finite subset $F$ of $A$ as in the conclusion of item (i). 
Let  $K$ be the closure of $\rhull{A\setminus F}$ in $L$.
Since $V$ is closed in $L$ and $\rhull{A\setminus F}\subseteq V$, the inclusion $K\subseteq V$ holds as well. Since $K$ has finite co-dimension in $L$, the quotient $L/K$ 
is finite-dimensional, so the quotient map $q: L \to L/K$ is continuous by Fact \ref{closed:kernels}. 
As $\dim L/K < \infty$, it carries a unique topological vector space topology, so 
$q$ is also open. Hence, $q(V)$ is a neighbourhood of $0$ in $L/K$. As $\dim L/K < \infty$, there exists a convex neighbourhood $U$ of $0$ in $L/K$ contained in $q(V)$. Then $q^{-1}(U)$ is a convex neighbourhood of $0$ in $L$ such that $q^{-1}(U) \subseteq q^{-1}(q(V)) = V + K \subseteq V + V \subseteq W$. 
\end{proof}

Let $A$ be a linearly independent subset of a vector space $E$. For every set $B\subseteq  A$ 
we denote by $\pi^A_B: \rhull{A}\to \rhull{B}$  the unique projection from $\rhull{A}$ onto $\rhull{B}$
such that $\ker \pi^A_B=\rhull{A\setminus B}$  and the restriction of $\pi^A_B$ to $\rhull{B}$ is the identity map of $\rhull{B}$.
For $b\in A$, we use $\pi^A_b$ instead of $\pi^A_{\{b\}}$ for simplicity.

\begin{lemma}
\label{lemma:continuity:of:projections}
Let $B$ be a non-empty finite subset of a linearly independent subset $A$ of a topological vector space $E$. Then $\pi^A_B$ is continuous if and only if $\pi^A_b$ is continuous for every $b\in B$. 
\end{lemma}
\begin{proof}
To check the ``only if part'', assume that $\pi^A_B$ is continuous and fix an arbitrary $b\in B$. Being a linear map defined on the finite-dimensional space $\rhull{B}$, $\pi^{B}_{b}$ is continuous, and so is the composition $\pi^{B}_{b}\circ \pi^A_B=\pi^A_{b}$.

Let us show the ``if'' part. Assume that $\pi^A_b$ is continuous for every $b\in B$. By Fact \ref{closed:kernels}, $\ker \pi^A_b$
is a closed subset of $E$ for every $b\in B$. Therefore, $\ker\pi^A_B=\bigcap_{b\in B} \ker \pi^A_b$ is a closed subset of $E$ as well. Therefore, 
$\pi^A_B$ is continuous by Fact \ref{closed:kernels}.
\end{proof}

The next example shows that even for a linearly independent, absolutely Cauchy summable subset $A$ of a topological vector space, {\em all\/} projections $\pi^A_B$ for a non-empty finite set $B\subseteq A$ can be discontinuous.

\begin{example}
\label{bad:projections} For each $i\in\N$, let  $e_i=(0,0,\ldots,1,0,\ldots)\in \R^\N$, where the only $1$ is at the $i$th place.  Define $a_0=e_0$ and $a_i=e_i-e_{i-1}$ for $i>0$.
Then {\em $A=\{a_i:i\in\N\}$ is a faithfully indexed (hence infinite) linearly independent, absolutely Cauchy summable subset of  $E=\rhull{A}=\R^{(\N)}$ such that the projection $\pi^A_B$ is discontinuous for each non-empty 
finite set $B\subseteq  A$.} All statements except discontinuity of projections are straightforward. 

By Lemma \ref{lemma:continuity:of:projections}, it suffices to  show that the projection $\pi^A_{a_k}$ is discontinuous for every $k\in\N$. Fix a $k\in\N$. It follows from the definition of $\pi^A_{a_k}$ that  
\begin{equation}
\label{eq:projections}
\pi^A_{a_k}(a_k)=a_k
\text{ and }
\pi^A_{a_k}(a_i)=0
\text{ for all }
i\in\N
\text{ with }
i\not=k.
\end{equation}

Let $n\in\N$ and $n\ge k$.
Note that $e_n=\sum_{i=0}^{n} a_i$ by the definition of $a_i$.
Combining this with the linearity of $\pi^A_{a_k}$, \eqref{eq:projections} and $n\ge k$, we get
\begin{equation}
\label{eq:limit}
\pi^A_{a_k}(e_n)=\sum_{i=0}^{n}\pi^A_{a_k}(a_i)=\pi^A_{a_k}(a_k)=a_k\not=0.
\end{equation}
Since $\lim_{n\to\infty} e_n=0$, yet $\lim_{n\to\infty}\pi^A_{a_k}(e_n)=a_k\not=0=\pi^A_{a_k}(0)$ by \eqref{eq:limit}, the map $\pi^A_{a_k}$ is discontinuous.
\end{example}

\begin{remark}
(i)
{\em If $A$ is an absolutely Cauchy summable subset of a topological vector space, then for every finite-dimensional subspace $F$ of $\rhull{A}$ there exists a continuous projection from $\rhull{A}$ onto $F$.}
Indeed, since $\rhull{A}$ is locally convex by Proposition \ref{R:Cauchy:summable}~(ii), there exists a closed subspace $L$ of $\rhull{A}$ such that 
$\rhull{A}=F\oplus L$ algebraically.\footnote{In fact, it follows from \cite[page 156, statement (2)]{Koethe}
that $\rhull{A}=F\times L$ holds topologically as well.} Therefore, the projection $p:\rhull{A}\to F$ with $\ker p=L$ is continuous by Fact \ref{closed:kernels}.

(ii) Let $A$ be the linearly independent, absolutely Cauchy summable subset of  $\R^{(\N)}$ constructed in Example \ref{bad:projections}.
Then for every  non-empty finite set $B\subseteq A$, the canonical projection $\pi^A_B:\rhull{A}\to\rhull{B}$ is discontinuous, yet by item (i), there exists {\em some\/} continuous projection from $\rhull{A}$ onto $\rhull{B}$.
\end{remark}

The following lemma is perhaps known. We include its proof only for the reader's convenience.

\begin{lemma}
\label{continuity:of:linear:functionals}
A linear functional $f: E\to\R$ on a topological vector space $E$ is continuous if and only if $f(U)\not=\R$ for some neighbourhood $U$ of $0$ in $E$.
\end{lemma}

\begin{proof}
The ``only if'' part is clear. 
To prove the ``if'' part, fix a neighbourhood $U$ of $0$ in $E$ with $f(U)\not=\R$.
Let $V$ be a balanced neighbourhood of $0$ in $E$ contained in $U$. Then $f(V)$ is a proper balanced subset of $\R$, so  it must be bounded; that is, $f(V)\subseteq (-r,r)$ for some real 
number $r>0$. Finally, for every $\varepsilon>0$, $W=\frac{\varepsilon}{r} V$ is a neighbourhood of $0$ in $E$ such that $f(W)= \frac{\varepsilon}{r} f(V)\subseteq (-\varepsilon,\varepsilon)$. This shows that $f$ is continuous.
\end{proof}

 Our next theorem produces infinite sets for which all internal projections are continuous.

\begin{theorem}
\label{subsets:with:continuous:projections}
Every infinite linearly independent, absolutely Cauchy summable subset $A$  of a topological vector space contains an infinite subset $B$ such that the projection $\pi^B_C:\rhull{B}\to \rhull{C}$ is continuous for every finite  set $C\subseteq B$. 
\end{theorem}
\begin{proof}
Define $B_{-1}=\emptyset$ and $A_{-1}=A$. By induction on $n\in\N$, we shall select  sets $B_n$ and  $A_n$ satisfying the following conditions:
\begin{itemize}
\item[(i$_n$)] $A_n$ is infinite;
\item[(ii$_n$)] $B_n$ is finite;
\item[(iii$_n$)] $B_n\subseteq A_n\subseteq A$;
\item[(iv$_n$)] $B_{n-1}\subseteq B_n$ and $B_n\setminus B_{n-1}\not=\emptyset$;
\item[(v$_n$)] $A_n\subseteq A_{n-1}$;
\item[(vi$_n$)] $\pi^{A_n}_b$ is continuous for every $b\in B_n$. (By 
Lemma \ref{lemma:continuity:of:projections}, 
 this is equivalent to the continuity of the map $\pi^{A_n}_{B_n}$.)
\end{itemize}

Suppose that $n\in\N$ and the sets $A_m$ and $B_m$ satisfying conditions 
(i$_m$)--(vi$_m$) have already been selected for all $m\in\N$ with $m<n$.  We shall define sets $A_n$ and $B_n$ satisfying conditions
(i$_n$)--(vi$_n$).

Since $A_{n-1}$ is infinite and $B_{n-1}$ is finite by 
(i$_{n-1}$) and (ii$_{n-1}$) respectively, we can select $c\in A_{n-1}\setminus (B_{n-1}\cup\{0\})$. Clearly, 
\begin{equation}
\label{eq:B_n}
B_n=B_{n-1}\cup\{c\}
\end{equation}
 is a finite set, so (ii$_n$) holds. Note that (iv$_n$) holds as well.

Since $c\not=0$, there exists a neighbourhood $O$ of $0$ such that $c\not\in O$. Choose a neighbourhood $V$ of $0$ such that $V-V-V\subseteq O$.
Since $A$ is absolutely Cauchy summable,  we can use Proposition \ref{R:Cauchy:summable}~(i) to find a finite set $F\subseteq A$ such that 
\begin{equation}
\label{eq:tails}
\rhull{A\setminus F}\subseteq V.
\end{equation}
By enlarging $F$ if necessary, we shall assume, without loss of generality, that $B_n\subseteq F$.  Define 
\begin{equation}
\label{eq:A_n}
A_n=(A_{n-1}\setminus F)\cup B_n.
\end{equation}
Since $B_{n-1}\subseteq A_{n-1}$ by (iii$_{n-1}$) and $c\in A_{n-1}$,  from \eqref{eq:B_n} we get $B_n\subseteq A_{n-1}$.  From this and \eqref{eq:A_n} we conclude that (v$_n$) holds.
Since $F$ is finite and $A_{n-1}$ is infinite by (i$_{n-1}$), \eqref{eq:A_n} implies that $A_n$ is infinite as well; that is, (i$_n$) holds. The condition (iii$_n$) follows from (iii$_{n-1}$) and \eqref{eq:A_n}.

It remains only to check the condition (vi$_n$).  Consider first the case when $b\in B_{n-1}$. Since $A_n\subseteq A_{n-1}$ by (v$_n$) and $\pi^{A_{n-1}}_b$ is continuous by (vi$_{n-1}$), so is its restriction  $\pi^{A_n}_b=\pi^{A_{n-1}}_b\restriction{\rhull{A_n}}$ to $\rhull{A_n}$. Since $B_n\setminus B_{n-1}=\{c\}$, it remains only to verify that the linear functional  $\pi^{A_n}_c$ is continuous. By Lemma \ref{continuity:of:linear:functionals}, 
it suffices to find  a neighbourhood  $U$ of $0$ such that
\begin{equation}
\label{missing:c}
c\not\in\pi^{A_n}_{c}(U\cap \rhull{A_n}).
\end{equation}

Since $\pi^{A_{n-1}}_{B_{n-1}}$ is continuous by (vi$_{n-1}$), 
there exists  a neighbourhood $W$ of $0$ such that 
\begin{equation}
\label{eq:small:nghb}
\pi^{A_{n-1}}_{B_{n-1}}(W\cap\rhull{A_{n-1}})\subseteq V.
\end{equation}

We claim that $U=V\cap W$ is the desired neighbourhood. Indeed, assume that \eqref{missing:c} fails. Then $c=\pi^{A_n}_{c}(d)$ for some $d\in U\cap \rhull{A_n}$. It follows from \eqref{eq:B_n} and \eqref{eq:A_n} that 
$A_n=(A_{n-1}\setminus F)\cup B_{n-1}\cup\{c\}$. Since these three sets are pairwise disjoint and $A$ is linearly independent, there exist unique $a\in \rhull{A_{n-1}\setminus F}$ and $b\in \rhull{B_{n-1}}$ such that
$d=a+b+c$. In particular,
\begin{equation}
\label{eq:b:as:projection}
b=\pi^{A_{n-1}}_{B_{n-1}}(d).
\end{equation}
To get a contradiction, it is enough to show that each of the three elements $d$, $a$, $b$ belongs to $V$. Indeed, assuming that this has already been proved,
we would get $c=d-a-b\in V- V-V\subseteq O$, in contradiction with our choice of $O$.

First, $d\in U\subseteq V$. Second, $d\in U\subseteq W$ and 
$d\in \rhull{A_{n}}\subseteq \rhull{A_{n-1}}$ by (v$_n$), so that 
$d\in W\cap\rhull{A_{n-1}}$. Combining this with 
\eqref{eq:small:nghb} 
and 
\eqref{eq:b:as:projection},
we conclude that $b\in V$. Finally, $a\in \rhull{A_{n-1}\setminus F}\subseteq \rhull{A\setminus F}\subseteq V$ by (iii$_{n-1}$) and \eqref{eq:tails}.
This finishes the verification of the condition (vi$_n$).

The inductive step has been completed.

Since (iii$_m$) holds for every $m\in\N$,
\begin{equation}
\label{eq:def:B}
B=\bigcup_{m\in\N} B_m
\end{equation}
 is a subset of $A$.  Since (iv$_m$) holds for every $m\in\N$, $B$ is infinite. We claim that 
\begin{equation}
\label{eq:B:in:the:intersection}
B\subseteq \bigcap_{n\in \N} A_n.
\end{equation}
Indeed, by 
\eqref{eq:def:B},
in order to establish \eqref{eq:B:in:the:intersection}, it suffices to check that $B_m\subseteq A_n$ for all $m,n\in\N$. If $m\le n$, then  $B_m\subseteq B_n$, as (iv$_k$) holds for every $k\in\N$. Since $B_n\subseteq A_n$ by (iii$_n$), this yields the inclusion $B_m\subseteq A_n$ for $m\le n$. Suppose now that $n< m$. By (iii$_m$), $B_m\subseteq A_m$. Since (v$_k$) holds for every $k\in\N$, we have $A_m\subseteq A_n$. This yields the inclusion $B_m\subseteq A_n$ in case  $n< m$ as well.

By Lemma \ref{lemma:continuity:of:projections}, to finish the proof of our theorem it remains only to show that $\pi^{B}_b$ is continuous for every $b\in B$.  Let $b\in B$ be arbitrary. By \eqref{eq:def:B}, there exists $m\in\N$ such that $b\in B_m$. The map $\pi^{A_m}_b$ is continuous by (vi$_m$). Since $B\subseteq A_m$ by \eqref{eq:B:in:the:intersection}, its restriction $\pi^{A_m}_b\restriction_{B}=\pi^{B}_b$ to $\rhull{B}$ is also continuous.
\end{proof}

\begin{question}
If $A$ is an absolutely Cauchy summable subset of a topological vector space, is $\rhull{A}$ topologically isomorphic to $\R^{(I)}$ for some index set $I$?
\end{question}

\section{The linear Kalton map and its connection to the Kalton map}
\label{Sec:lkal}
Let $E$ be a vector space and let $A\subseteq E\setminus\{0\}$. Then there exists a unique linear map
\begin{equation}\label{eq:def:of:lkal}
\lkal{A}:\bigoplus_{a\in A}\rhull{a}\to E
\end{equation}
which extends each natural inclusion map $\rhull{a}\to E$ for $a\in A$. We call the map $\lkal{A}$ as in \eqref{eq:def:of:lkal} the {\em linear Kalton map associated with $A$}. Clearly, the set $A$ is linearly independent (or equivalently, $A$ is a Hamel basis of $\rhull{A}$) if and only if $\lkal{A}$ is  injective. The vector space 
 $$
 lS_A=\bigoplus_{a\in A}\rhull{a}
 $$
is topologically isomorphic to the topological vector space $\R^{(A)}$. Moreover,  $lS_A$ contains $S_A$ as a subgroup and $\lkal{A}\restriction_{S_A}=\kal{A}$.

When $E$ is a topological vector space, one can discuss topological properties of the linear Kalton map, as now both the domain and range of the map are topological spaces. Proposition \ref{kal:iso:lkal:cont} deals with the continuity of this map, while Proposition \ref{openess:of:linear:kalton:map} deals with its openness.

\begin{proposition}\label{kal:iso:lkal:cont}
Let $E$ be a topological vector space and let $A\subseteq E\setminus\{0\}$. Then 
the following statements are equivalent:
\begin{itemize}
\item[(i)] The linear Kalton map $\lkal{A}$ is continuous.
\item[(ii)] The Kalton map $\kal{A}$ is continuous.
\item[(iii)] The set $A$ is absolutely Cauchy summable.
\end{itemize}
\end{proposition}
\begin{proof} 
(i) implies (ii), as $\lkal{A}\restriction_{S_A}=\kal{A}$. Items (ii) and (iii) are equivalent by Theorem \ref{Proposition:Udine}. 

(iii)~$\to$~(i)
Let $U$ be a neighbourhood of $0$ in $E$. Choose a neighbourhood $V$ of $0$ in $E$ such that $V+V\subseteq U$. By (iii) and Proposition \ref{R:Cauchy:summable}~(i),  there exists 
a finite set $F\subseteq A$ such that $\rhull{A\setminus F}\subseteq V$.  If $F=\emptyset$, we let $W=V$. Otherwise
we fix a neighbourhood $W$ of zero of $E$ such that   $W+W+\dots+W\subseteq V$, where $|F|$-many $W$'s are taken in the sum. Then 
$$
\lkal{A}\left( \bigoplus_{a\in F} (\rhull{a} \cap W)\oplus\bigoplus_{a\in A\setminus F} \rhull{a}\right)
=
\sum_{a\in F} (\rhull{a} \cap W)+\rhull{A\setminus F}\subseteq V+V\subseteq U.
$$
This establishes the continuity of $\lkal{A}$.
\end{proof}

\begin{proposition}
\label{openess:of:linear:kalton:map}
For a linearly independent subset $A$ of a topological vector space $E$, the following conditions are equivalent:
\begin{itemize}
\item[(i)] the linear Kalton map $\lkal{A}: lS_A\to \rhull{A}$ is an open map onto its image $\rhull{A}$;
\item[(ii)] the linear functional $\pi^A_a: \rhull{A}\to\rhull{a}$ is continuous for every $a\in A$.
\end{itemize}
\end{proposition}

\begin{proof}
For every $a\in A$, denote by $p_a$ the projection from $lS_A$ to the $a$th coordinate $\rhull{a}$ of the direct sum $lS_A$. Clearly,
\begin{equation}
\label{commutative:diagram}
\lkal{A}\restriction_{\rhull{a}}\circ p_a = \pi^A_a\circ \lkal{A}
\ 
\mbox{ for every }
\ 
a\in A.
\end{equation}
Recall that each $\lkal{A}\restriction_{\rhull{a}}$ is a topological isomorphism. It follows from the definition of $lS_A$ that each projection $p_a: lS_A\to \rhull{a}$ is continuous.

(i)$\to$(ii) Fix $a\in A$ and let $V$ be an open subset of $\rhull{a}$. It suffices to show that $(\pi^A_a)^{-1}(V)$ is an open subset of $\rhull{A}$.
Since $\lkal{A}\restriction_{\rhull{a}}^{-1}$ is a topological isomorphism, $\lkal{A}\restriction_{\rhull{a}}^{-1}(V)$ is open in $\rhull{a}$. Since $p_a$ is continuous, $W = 
p_a^{-1}(\lkal{A}\restriction_{\rhull{a}}^{-1}(V))$ is open in $lS_A$.
Applying (i), we conclude that $\lkal{A}(W)$ is open in 
$\lkal{A}(lS_A)=\rhull{A}$.
From \eqref{commutative:diagram}, we get
\begin{equation}
\label{eq:29}
(\pi^A_a)^{-1}(V)=\lkal{A}\left((\lkal{A}\restriction_{\rhull{a}}\circ p_a)^{-1}(V)\right)=\lkal{A}(W).
\end{equation}
Therefore, $(\pi^A_a)^{-1}(V)$ is an open subset of $\rhull{A}$.

(ii)$\to$(i) 
To establish (i), it suffices to find, for every basic open neighbourhood $O$ of $0$ in $lS_A$, an open neighbourhood $U$ of $0$ in $E$ such that $U\cap \rhull{A}\subseteq \lkal{A}(O)$. There exist a finite set $F\subseteq A$ and a family $\{W_a:a\in F\}$ of open neighbourhoods of $0$ in $E$ such that
\begin{equation}
\label{eq:O}
O=\bigoplus_{a\in F} (\rhull{a} \cap W_a)\oplus\bigoplus_{a\in A\setminus F} \rhull{a}
\end{equation}
For every $a\in F$ we can use (ii) to fix an open neighbourhood $V_a$ of $0$ in $E$ such that $\pi^A_a(V_a)\subseteq W_a$. Since $F$ is finite, $U=\bigcap_{a\in F} V_a$ is an open neighbourhood of $0$ in $E$. Since $\pi^A_a(U)\subseteq W_a$ for every $a\in F$, from \eqref{eq:O} and the definition of $\lkal{A}$ we conclude that
$U\cap \rhull{A}\subseteq \lkal{A}(O)$.
\end{proof}

\begin{example}
Let $A$ be the linearly independent, absolutely Cauchy summable subset of $\R^{(\N)}$ constructed in Example \ref{bad:projections}.
Then {\em the linear Kalton map $\lkal{A}$ is a continuous injection which is not an open map onto its image $\rhull{A}$\/}.
Indeed, the continuity of $\lkal{A}$ follows from Proposition \ref{kal:iso:lkal:cont}. Since the canonical projections $\pi^A_a$ are discontinuous, 
$\lkal{A}$ is not open by Proposition \ref{openess:of:linear:kalton:map}.
\end{example}

This example shows that in our next theorem one cannot take $B=S$.

\begin{theorem}
\label{abs:Cauchy:summable:contain:isomorphic:subsets}
Every infinite absolutely Cauchy summable subset 
$S$ of a topological vector space 
contains an infinite subset $B$ 
such that 
the linear Kalton map 
$\lkal{B}$ 
is a topologically isomorphic embedding. 
\end{theorem}
\begin{proof} 
Let $A$ be a
maximal independent subset of $S$ which exists by Zorn's Lemma.
Then $\rhull{S}=\rhull{A}$ by the maximality of $S$. If $A$ were finite, then
$S$ would become an infinite Cauchy summable subset of the finite-dimensional Euclidean space $\rhull{A}$. Since finite-dimensional Euclidean spaces are NSS groups, this would contradict Theorem \ref{metric:NSS:iff:no:abs:C:summable}. Therefore, $A$ must be infinite.

As a subset of the absolutely Cauchy summable set $S$, $A$ itself is absolutely Cauchy summable. Since $A$ is also linearly independent, we can use 
Theorem \ref{subsets:with:continuous:projections} to choose an infinite subset $B$ of $A$ as in the conclusion of this theorem.

As a subset of the linearly independent set $A$, $B$ is also linearly independent. Therefore, the linear Kalton map $\lkal{B}$ is an injection.
Furthermore, $\lkal{B}$ is an open map onto its image $\rhull{B}$ by Proposition \ref{openess:of:linear:kalton:map}.
As a subset of the absolutely Cauchy summable set $S$, $B$ is absolutely Cauchy summable as well.
Therefore, the linear Kalton map $\lkal{B}$ is continuous by the implication (iii)~$\to$~(i) of Proposition \ref{kal:iso:lkal:cont}.
\end{proof}

\begin{corollary}
\label{continuous:contain:isomorphic:subsets}
Let $S$ be an infinite subset of a topological vector space such that the Kalton map $\kal{S}$ is continuous. Then $S$ contains a countably infinite subset $B$ such that the linear Kalton map $\lkal{B}$ is a topologically isomorphic embedding. 
\end{corollary}
\begin{proof}
By Proposition \ref{openess:of:linear:kalton:map}, $S$ is absolutely Cauchy summable. Now the conclusion follows from Theorem \ref{abs:Cauchy:summable:contain:isomorphic:subsets}.
\end{proof}

If $A$ is a subset of a topological vector space such that the Kalton map $\kal{A}$ is topologically isomorphic embedding, then the set $A$ is topologically independent (Proposition \ref{Proposition:Matsuyama}) and consequently, it is also linearly independent by Proposition \ref{prop:step:to:Lie}~(i). Thus the linear Kalton map $\lkal{A}$ is an isomorphism onto $\rhull{A}$. Moreover, $\lkal{A}$ is also continuous by Proposition \ref{kal:iso:lkal:cont}. However, we do not know whether it is an open map onto its image $\rhull{A}$.

 \begin{question}
If $A$ is a subset of a topological vector space such that $\kal{A}$ is a topologically isomorphic embedding, is then $\lkal{A}$ 
a topologically isomorphic embedding as well?
\end{question}

\section{Infinite direct sums and products in topological vector spaces}\label{Sec:final}

In this section we provide characterizations of topological vector spaces that contain either $\R^{(\N)}$ or $\R^\N$ as a subspace. As particular corollaries we obtain some classical results from \cite{BPR,Lip}.

\begin{theorem}\label{complete:TVS:not:TAP:iff:contains:R:to:N}
Let $E$ be a topological vector space. Then the following conditions are equivalent:
\begin{itemize}
\item[(i)] $E$  contains the topological vector space $\R^{(\N)}$ as its subspace;
\item[(ii)] $E$ contains  the topological group $\Z^{(\N)}$ as its subgroup;
\item[(iii)] $E$ has an infinite absolutely Cauchy summable set.
\end{itemize}
 Moreover, if $E$ is complete, then these three conditions are also equivalent to: 
\begin{itemize}
\item[(iv)]
$E$ contains the topological vector space $\R^\N$ as its subspace.
\end{itemize} 
 \end{theorem}

\begin{proof}  The implication (i)~$\to$~(ii) is clear.

The implication (ii)~$\to$~(iii) follows from Lemma \ref{Cauchy:summable:in:direct:sum} and Remark \ref{remark:on:abs:C:summability}~(ii).

The implication (iii)~$\to$~(i) follows from Theorem 
\ref{abs:Cauchy:summable:contain:isomorphic:subsets}.

The equivalence  (i)~$\leftrightarrow$~(iv)  for a complete topological vector space $E$ follows from the uniqueness of the completion and the fact that $\R^\N$ is the completion of $\R^{(\N)}$.
\end{proof}

We say that  a topological vector space $E$ {\em has small lines\/} if,  for every neighbourhood $V$ of $0$ in $E$ there exists $x\in E\setminus\{0\}$ with 
$\rhull{x}\subseteq V$; that is, $V$ contains a ``line" (one-dimensional subspace). 

\begin{corollary}\cite[Theorem 3]{Lip}
\label{Lip:corollary}
For a metric vector space $E$, the following conditions are equivalent:
\begin{itemize}
\item[(i)] $E$ contains a subspace topologically isomorphic to $\R^{(\N)}$;
\item[(ii)] $E$ has small lines.
\end{itemize}
\end{corollary}

\begin{proof}
The implication (i)~$\to$~(ii) follows from the fact that $\R^{(\N)}$ has small lines. To establish the implication (ii)~$\to$~(i), assume that $E$ has small lines. Then
$E$ is not NSS. By Theorem \ref{metric:NSS:iff:no:abs:C:summable},  $E$ contains an infinite absolutely Cauchy summable set. To get (i) it remains to apply Theorem \ref{complete:TVS:not:TAP:iff:contains:R:to:N}.
\end{proof}

Combining Theorem \ref{complete:TVS:not:TAP:iff:contains:R:to:N} with Corollary \ref{Lip:corollary}, one obtains the following classical result of Bessaga, Pelczynski and Rolewicz:

\begin{corollary}\cite[Theorem 9]{BPR}
A complete metric vector space contains a subspace topologically isomorphic to $\R^{\N}$ if and only if it  has small lines.
\end{corollary}

Recall that in the multiplier convergence theory one typically takes a fixed set $\mathscr{F}\subseteq \R^\N$ of ``multipliers'' and calls a
series $\sum_{n=0}^\infty a_n$ in a topological vector space {\em $\mathscr{F}$ multiplier convergent\/} provided that the series $\sum_{n=0}^\infty f(n)a_n$ converges
for every $f\in\mathscr{F}$. 
By varying the set $\mathscr{F}$ of multipliers one can obtain a fine description of the level of convergence
of a given series, leading to a rich theory; see \cite{Swartz}. The toughest convergence condition on a series is obviously imposed by taking 
$\mathscr{F}$ to be the whole $\R^\N$. For brevity, we shall say that a sequence $\{a_n:n\in\N\}$ of elements of a topological vector space $E$ is {\em $\R^\N$-convergent\/}
 if the series $\sum_{n=0}^\infty a_n$ is $\R^\N$ multiplier convergent; that is, if the series $\sum_{n=1}^\infty r_n a_n$ converges to some element of $E$ for every real sequence  $\{r_n:n\in\N\}$.
For such a sequence $A= \{a_n:n\in\N\}$ the map $\N \to A$, defined by $n\mapsto a_n$ may fail to be finitely-many to one only when $a_n=0$ for all but finitely many members of the sequence. In such a case we say that $A= \{a_n:n\in\N\}$ is {\em trivial\/}, otherwise (if infinitely many $a_n\ne 0$) we call it {\em non-trivial\/}. 

Bessaga, Pelczynski and Rolewicz used the notion of an $\R^\N$-convergent sequence (without giving it a name) to obtain a characterization of complete metric spaces containing a subspace isomorphic to $\R^\N$. Indeed, they proved that a complete metric vector space contains a subspace isomorphic to $\R^\N$ if and only if it contains 
a non-trivial $\R^\N$-convergent sequence
\cite[Corollary]{BPR}.
Our last theorem shows that both ``metric'' and ``complete'' are superfluous in this result.

\begin{theorem}\label{another:theorem}
A topological vector space contains a subspace isomorphic to $\R^\N$ if and only if it contains 
a non-trivial $\R^\N$-convergent sequence.
\end{theorem}

\begin{proof}
The ``only if'' part is obvious. To show the  ``if'' part, assume that $A=\{a_n:n\in\N\}$ is a non-trivial
$\R^\N$-convergent sequence in a topological vector space $E$. Then the {\em set\/} $A$ must be infinite and absolutely summable. 
Consequently, the set $A$ is absolutely Cauchy summable by Proposition \ref{as:is:acs}. 
By Theorem \ref{abs:Cauchy:summable:contain:isomorphic:subsets}, there is an infinite faithfully indexed subset $B=\{b_n:n\in\N\}$ of $A$ such that the linear Kalton map $\lkal{B}$ is a topologically isomorphic embedding. By the uniqueness of completions of topological vector spaces, its continuous extension $\slkal{B}: \overline{lS_B} \to \overline{E}$ to the completions $\overline{lS_B}$ and $\overline{E}$ of $lS_B$ and $E$, respectively, is still a topological isomorphism. Clearly, $ \overline{lS_B}\cong \R^\N$.  Therefore, it suffices to show that $\slkal{B}(\overline{lS_B})\subseteq E$.

Since the faithfully indexed sequence $\{b_n:n\in\N\}$ is a subsequence of the $\R^\N$-convergent sequence $\{a_n:n\in\N\}$, the former sequence is $\R^\N$-convergent as well. 
Let $x_n=(\slkal{B})^{-1}(b_n)$ for every $n\in\N$. Since every element $v \in \overline{lS_B}$ can be written in the form $v  = \sum_{n=1}^\infty r_n  x_n$, we obtain
$$
\slkal{B}(v) = \slkal{B}\left(\sum_{n=1}^\infty r_nx_n\right)
=
\sum_{n=1}^\infty r_n\slkal{B}(x_n)
=
\sum_{n=1}^\infty r_nb_n
 \in E,
$$
which implies $\slkal{B}(\overline{lS_B})\subseteq E$.
\end{proof}

\begin{remark}\label{answer:to:a:question:of:M:Husek} Topological vector spaces without $\R^\N$-convergent sequences  were studied in \cite{Husek2} under the name ``spaces without free sums''.  In Proposition 4 of \cite{Husek2} 
Hu\v{s}ek proved that the following statements are equivalent for a topological vector space $E$:
\begin{enumerate}
 \item There exists a continuous linear mapping $f:\R^\N\to E$ such that $f(\R^\N)$ has infinite dimension.
 \item There exists a continuous linear mapping $f:\R^\N\to E$ such that infinitely many $f(e_i)$ are non-zero (here $e_i(j)=\delta_{ij}$, where $\delta_{ij}$ is the Kronecker's delta).
 \item $E$ contains an $\R^\N$-convergent sequence.
\end{enumerate} 
In a remark after \cite[Proposition 4]{Husek2} the author says that he is not aware  whether the following can be added to the above list of equivalent properties. 
\begin{itemize}
\item[(4)] There exists a one-to-one continuous linear mapping from $\R^\N$ to $E$.  
\end{itemize} 

Theorem \ref{another:theorem} shows that it is possible, thereby answering the question of Hu\v{s}ek. In fact, the same theorem shows that  one can add to the list of equivalent properties even the following stronger property:
\begin{itemize}
 \item[(5)] There exists a topologically isomorphic embedding $\R^\N \hookrightarrow E$.  
\end{itemize} 
\end{remark}

To the best of our knowledge $\R^\N$-convergent sequences were also studied in \cite{Pfister}.

\end{document}